\documentclass[]{interact}

\usepackage{color}
\usepackage{float}
\usepackage{tikz}
\usepackage{epstopdf}
\usepackage{graphicx}
\usepackage{makecell}
\usepackage{relsize}
\usepackage{wrapfig}

\usepackage{amsmath}
\usepackage{amsfonts}
\usepackage{amssymb}
\usepackage{caption}
\usepackage{subcaption}
\usepackage{setspace}
\usepackage[bottom]{footmisc}
\usepackage{tipa}
\usepackage{lscape}
\usepackage{url}
\usepackage{mathtools}
\usepackage{bm}
\usepackage{lscape}
\usepackage[ruled,linesnumbered]{algorithm2e}

\usepackage[natbibapa,nodoi]{apacite}
\theoremstyle{plain}
\newtheorem{theorem}{Theorem}[section]
\newtheorem{lemma}[theorem]{Lemma}

\theoremstyle{definition}

\theoremstyle{remark}

\begin{document}


\title{A Hybrid Inverse Optimization-Stochastic Programming Framework for Network Protection}

\author{\name{Stephanie Allen\textsuperscript{a}, Daria Terekhov\textsuperscript{b}, and Steven A. Gabriel \textsuperscript{c}
}
\affil{\textsuperscript{a} University of Maryland, College Park, \url{sallen7@umd.edu} 
}
\affil{\textsuperscript{b} Concordia University, \url{daria.terekhov@concordia.ca} 
}
\affil{\textsuperscript{c} University of Maryland, College Park \& Norwegian University of Science and Technology, \url{sgabriel@umd.edu} 
}
}

\maketitle

\begin{abstract}
Disaster management is a complex problem demanding sophisticated modeling approaches.  We propose utilizing a hybrid method involving inverse optimization to parameterize the cost functions for a road network's traffic equilibrium problem and employing a modified version of \cite{fan2010solving}'s two-stage stochastic model to make protection decisions using the information gained from inverse optimization.  We demonstrate the framework using two types of cost functions for the traffic equilibrium problem and show that accurate parameterizations of cost functions can change spending on protection decisions in most of our experiments.
\end{abstract}


\section{Introduction}

\noindent Given the threat of natural disasters, it is imperative that communities and nations prepare in order to mitigate the consequences.  According to NOAA \cite{noaa_centers}, in the year 2020, there were 22 ``weather and climate disasters'' that cost 1 billion or more US dollars in the United States, and 262 people died in these disasters.  Governments and planning agencies often have little foresight of the type of disaster that might strike, meaning they must be prepared for many different potential events.  To be effective, governments must make strategic decisions before and immediately after crises such that financial and human costs are minimized.  

As a response to these events, researchers have developed disaster support systems (DSS) which are defined by \cite{wallace1985decision} as systems which contain ``a data bank, a data analysis capability, normative models, and technology for display and interactive use of the data and models.'' We focus on the data analysis and normative model elements of these requirements which, respectively, mathematically examine and present information for decision makers and help make decisions.  One type of normative model used in the disaster management community which we employ in this paper is a multi-stage stochastic program.  This type of model is proposed as a way to make decisions regarding protecting networks against disasters or bringing supplies to communities after disasters.  The misspecification of the parameters defining this kind of model can severely limit the usefulness of the model.  In particular, we focus on transportation cost parameters, whose misspecification could lead to incorrect protection decisions, such as allocating too few or too many resources to parts of road networks affected by landslides and flash floods [\cite{WFP_link}].  
To our knowledge, both the DSS literature and the multi-stage stochastic program for disaster relief literature have not explored inverse optimization as a tool in estimating model parameters. \textit{We propose inverse optimization as a new approach for data analysis and demonstrate its ability to recover similar protection decisions as the originally parameterized stochastic network protection model.  We also demonstrate that accurate knowledge regarding the cost functions matters because it can change protection decisions when compared to the assumption of uniform cost for most of our experiments.}

The rest of the paper is organized as follows. Section \ref{sec:literature_review} investigates the literature related to our problem.  Section \ref{framework-explanation} provides appropriate mathematical background.  Section \ref{experimental-study} explains the experimental structure and the results.  Section \ref{conclusions} discusses our conclusions and ideas for future work.

\section{Literature Review}\label{sec:literature_review}

The literature in this section demonstrates that multi-stage disaster relief models and DSS models have not used inverse optimization in their frameworks before.

\subsection{Inverse Optimization for Transportation Problems}\label{lit_for_inverse_optimization}

Inverse optimization (IO) allows a user to parameterize particular functions in optimization and equilibrium problems using solutions to these problems [\cite{ahuja2001inverse}, \cite{bertsimas2015data}, \cite{zhang2011inverse}].  We focus on parameterizing the cost functions of a traffic equilibrium model.  \cite{thai2015multi} use a mathematical program with equilibrium constraints to minimize the difference between the simulated solutions and optimal solutions to the traffic equilibrium problem as a way of recovering the specified cost function parameters.  \cite{thai2018imputing} use a combination of methods by \cite{bertsimas2015data} and \cite{chen1998congested} to create a multi-objective program that minimizes the duality gap for the variational inequality and the difference between the optimal and observed solutions.  \cite{bertsimas2015data} use their inverse variational inequality problem along with kernel methods to estimate the cost functions.  \cite{zhang2017data} and \cite{zhang2018price} follow \cite{bertsimas2015data}'s methodology, with \cite{zhang2017data} involving different categories of vehicles and \cite{zhang2018price} emphasizing recovering both cost function and origin-destination matrices from real-world traffic data.
\cite{chow2014nonlinear} use techniques from \cite{ahuja2001inverse} but augment them to handle the nonlinear nature of their problem.  Finally, \cite{allen2021using} extend \cite{ratliff2014social}'s parameterization framework for multi-player Nash problems to the case of jointly convex generalized Nash equilibrium problems and demonstrate this by parameterizing a transportation game.

\subsection{Multi-Stage Disaster Relief Models}\label{lit_for_multi_stage}

There is a substantial literature on multi-stage stochastic programs for disaster relief and protection; see \cite{grass2016two}. Many methods for estimating cost functions in road networks include fuzzy numbers, Euclidean distances, road distance data, the Bureau of Public Roads (BPR) function (see Section \ref{subsection_cost_functions}), and stochastic parameters (referred to also as stochastic costs/scenario dependent costs).  None of them use inverse optimization to estimate costs, which is what we propose in this paper.  

First, \cite{zheng2013emergency} estimate cost parameters for moving supplies after natural disasters using fuzzy numbers for the time it takes to traverse between the supply and demand nodes.  
Second, \cite{barbarosoglu2004two} use road information and Euclidean distances between points for their cost parameters in their stochastic model pertaining to distributing supplies after natural disasters. \cite{chu2016optimization} also calculate the travel cost for several routes/paths of origin-destination pairs to capture the idea that one or more routes could fail in a disaster in their stochastic network protection model.  Travel cost is measured by a variable which takes on real numbers between 0 and 1 and which measures how close to the shortest path the demand for an OD pair is allowed to take through the network.
\cite{noyan2012risk} and \cite{doyen2019integrated} use a mixture of road data along with scenario dependent costs to form their cost functions, while \cite{mohammadi2016prepositioning} use exclusively scenario dependent costs. 

\cite{fan2010solving} employ the BPR function for their stochastic network protection problem for arc costs, but no stochastic parameters are involved.  For the rest of the BPR literature, either the capacity is impacted by the protection decisions and/or some of the parameters in the BPR function are stochastic [\cite{lu2018mean,lu2016bi,faturechi2014travel,asadabadi2017optimal,faturechi2018risk}].  For these multi-stage stochastic programs, inverse optimization methods would have captured a set of parameters that led to given flow patterns as seen in \cite{allen2021using} which could have been used to augment the existing cost function approaches.

\subsection{Disaster Support Systems}\label{lit_for_DSS}

We focus on reviewing DSSes that have a data analysis step in their processes.  
First, there are disaster support system papers that determine important quantities and parameters via simulation and/or physical models of the situation [\cite{alvear2013decision},\cite{cuesta2014real}, \cite{sahebjamnia2017hybrid},\cite{eguchi1997real}, \cite{fikar2016decision}, \cite{todini1999operational}, \cite{yilmaz2019finding}, \cite{kureshi2015towards}, \cite{yang2019emergency}, \cite{van2005decision}].  Second, DSS papers can also determine parameters via data processing as in \cite{fertier2020new}, \cite{horita2015development}, and \cite{zhang1994real}, or they can utilize machine learning to determine modeling structures as in \cite{abpeykar2014supervised}.  Other DSS papers use geographic information system (GIS) techniques to estimate parameters [\cite{rodriguez2018disaster,cioca2007spatial}]. \cite{horita2013approach} combine GIS and sensor information to estimate parameters.  In addition, some papers propose data fusion techniques such as ensemble Kalman filters [\cite{otsuka2016estimating}] and gradient based methods [\cite{kaviani2015decision}].  Inverse optimization allows a user to propose a model of the system and parameterize the model using data/simulated solutions\footnote{See \cite{tai2013identifying} for an example using proposed solutions.} and optimality conditions [\cite{ahuja2001inverse,chan2019inverse,bertsimas2015data,zhang2011inverse}], which can augment the information gained from data and, thus, could be useful for DSSes.  However, as can be seen from this review, DSSes have not used inverse optimization for data analysis.

\section{Hybrid Framework}\label{framework-explanation}

In the next few sections, we will explain our data analysis component (which presents information for decision makers) along with our normative model (which helps make decisions). The data analysis component centers on \cite{bertsimas2015data}'s work on parameterizing cost functions for the traffic equilibrium problem.  The normative model comes from \cite{fan2010solving} who suggest a two-stage network protection problem with equilibrium constraints to make protection decisions for road networks.  We propose pairing the two components together in the following sequence of steps, with $\theta$ representing the collection of parameters to be estimated by the inverse optimization model: 

\begin{enumerate}
    \item Input data $\hat{\mathbf{x}}^j,\ j=1...J$ into inverse optimization model (\ref{basic_bertsimas_formulation}) and obtain $\theta$ 
    \item Form stochastic network protection problem (SNPP) (\ref{SNPP_formulation}) with $\theta$ 
    \item Solve the SNPP (\ref{SNPP_formulation}) and obtain protection decisions $\mathbf{u}$
\end{enumerate}

\subsection{Data Analysis Component: Inverse Optimization}\label{sec:inverse-opt}


\cite{bertsimas2015data} utilize variational inequalities (VI) to represent optimization and equilibrium problems.  
\cite{bertsimas2015data} assume that the following extended variational inequality describes the $\epsilon$ equilibrium of a system, with $F:\mathbb{R}^n \rightarrow \mathbb{R}^n$, $\mathcal{F} \subset \mathbb{R}^n$, $\mathbf{x}\in \mathcal{F}$, and $\epsilon \in \mathbb{R}_+$:

\begin{equation}\label{VI_epsilon}
F(\mathbf{x}^*)^T (\mathbf{x} - \mathbf{x}^*) \geq -\epsilon,\ \forall \mathbf{x} \in \mathcal{F}.
\end{equation}

\noindent When $\epsilon=0$, we recover the classical VI.  In the case of our traffic application, we assume that $F$ is the cost function, representing the time per vehicle along each arc in the set $\mathcal{A}$ of arcs [\cite{lu2018mean,lu2016bi}].  Therefore, expression (\ref{VI_epsilon}) states that $\mathbf{x}^*$ solves the VI if the inner product between $F$ at $\mathbf{x}^*$ and the difference between any point in $\mathcal{F}$ and $\mathbf{x}^*$ is greater than a small, negative number.

\cite{bertsimas2015data} describe the Wardrop traffic equilibrium with nodes $\mathcal{N}$ and arcs $\mathcal{A}$ as having a node-arc incidence matrix $N \in \{-1,0,1\}^{|\mathcal{N}|\times |\mathcal{A}|}$ [\cite{marcotte2007traffic}], vectors $\mathbf{d}^w$ which contain the origin destination locations represented by the set $\mathcal{W}$ with a negative entry for the origin and a positive entry for the destination [\cite{marcotte2007traffic}], and a feasible set $\mathcal{F}$.  \cite{bertsimas2015data} define the $\mathcal{F}$ set as:\footnote{We define the $\mathcal{F}$ set for the inverse optimization model differently; see Appendix A.2.}
\begin{equation}\label{F_set_for_TEP}
    \mathcal{F} = \left\{ \mathbf{x} : \exists \mathbf{x}^w \in \mathbb{R}^{|\mathcal{A}|}_+\  s.t.\  \mathbf{x} = \sum\limits_{w\in \mathcal{W}} \mathbf{x}^w,\ N\mathbf{x}^w = \mathbf{d}^w \ \forall w \in \mathcal{W} \right\}
\end{equation}

\noindent in which $\mathbf{x}^w \in \mathbb{R}^{|\mathcal{A}|}_+$ represents the flow between origin and destination $w$ and $\mathbf{x} \in \mathbb{R}^{|\mathcal{A}|}_+$ represents the composite flow vector.  The corresponding $F$ function for the variational inequality is defined as $\mathbf{c}(\mathbf{x})$ such that $c_a:\mathbb{R}^{|\mathcal{A}|}_+ \rightarrow \mathbb{R}_+$ for arc $a$. 

The multipliers associated with the constraints in the $\mathcal{F}$ set should be non-negative because they represent the time it takes to travel from the associated node to the destination $w$ [\cite{fan2010solving}].  Therefore, we need to turn the equalities in the $\mathcal{F}$ set into inequalities [\cite{ban2005quasi,ban2006general}].  Respecting the definition of $N$ and $\mathbf{d}^w$, our traffic equilibrium problem in complementarity form is then [\cite{bertsimas2015data,sheffi_urban_transport,fan2010solving,gabriel2012complementarity,ban2006general,ban2005quasi}]:\footnote{We keep the row in $N$ that contains the destination, which is different from \cite{ban2005quasi} and \cite{ban2006general}.}

\begin{subequations}\label{traffic_equilibrium_problem}
\begin{equation}\label{TEP:cost}
    0 \leq \mathbf{c}(\mathbf{x}) + N^T \mathbf{y}^w \ \bot \ \mathbf{x}^w \geq 0,\  \forall w \in \mathcal{W}
\end{equation}
\begin{equation}\label{TEP:demand}
     0 \leq \mathbf{d}^w - N \mathbf{x}^w\  \bot\  \mathbf{y}^w \geq 0,\ \forall w \in \mathcal{W}
\end{equation}
\end{subequations}

\noindent We can show that $\mathbf{d}^w - N \mathbf{x}^w = 0$ when there is a solution for (\ref{traffic_equilibrium_problem}) and when we assume that the $\mathbf{c}(\mathbf{x})$ function is greater than 0 for all $\mathbf{x}\geq 0$ in a proof which is adapted from \cite{ban2005quasi}.  See the Appendix A.1.  We use (\ref{traffic_equilibrium_problem}) to generate data for the inverse optimization part of the framework, and the data generation process can be found in Section \ref{simulation_structure_section}.

For the inverse optimization model in \cite{bertsimas2015data}, we can then form an optimization model including each data point $\hat{\mathbf{x}}^j$ $j=1,...,J$ (with $J$ representing the total number of data points) representing the flow on the network such that: 

\begin{itemize}
\item There is one OD pair for each instance $\hat{\mathbf{x}}^j$. 

\item There is the same node-arc incidence matrix $N$ for each $\hat{\mathbf{x}}^j$.
\end{itemize}

\noindent The inverse optimization model for $J$ data points (corresponding to each of the $\hat{\mathbf{x}}^j$ flow patterns), parameters $\theta \in \Theta$ with $\Theta$ as a convex subset of $\mathbb{R}^{Z}$ ($Z$ representing a number of parameters), $\mathbf{y}^j \in \mathbb{R}^{|\mathcal{N}|}$, and $\epsilon \in \mathbb{R}^{J}$ is:  

\begin{subequations}\label{basic_bertsimas_formulation}
\begin{equation}
    \min\limits_{\theta \in \Theta,\mathbf{y},\epsilon} ||\epsilon||^2
\end{equation}
\begin{equation}
    -(N)^T \mathbf{y}^j \leq \mathbf{c}(\hat{\mathbf{x}}^j;\theta),\ j=1,...,J,
\end{equation}
\begin{equation}
    \mathbf{y}^j \geq 0,\ j = 1,...,J,
\end{equation}
\begin{equation}
    \mathbf{c}(\hat{\mathbf{x}}^j;\theta)^T \hat{\mathbf{x}}^j + (\mathbf{d}^j)^T \mathbf{y}^j \leq \epsilon^j,\ j=1,...,J,
\end{equation}
\end{subequations}

\noindent The derivation of this mathematical program can be found in Appendix A.2.  We solve this mathematical program using the \url{ipopt} solver [\cite{wachter2006implementation}].  There can be multiple forms for the vector-valued arc cost function $\mathbf{c}(\mathbf{x};\theta)$, which the next subsection will cover.

\subsubsection{Different Types of Cost Functions}\label{subsection_cost_functions}

We propose two different formulations for the vector valued function $\mathbf{c}(\mathbf{x})$ which represents the time per vehicle along each arc in the set of arcs $\mathcal{A}$ [\cite{lu2018mean,lu2016bi}].  Note that $\theta$ will represent the collection of all parameters for a given function. 

\begin{itemize}
\item \textbf{Linear Cost:} We assume that $c_a(\mathbf{x}) = \phi_a \mathbf{x}_a + \beta_a,\ \phi_a \in \mathbb{R}_+,\ \beta_a \in \mathbb{R}_+$, such that 
\begin{equation}
\mathbf{c}(\mathbf{x}) = \begin{bmatrix} \phi_1 & & \\ & \ddots & \\ & & \phi_{|\mathcal{A}|} \end{bmatrix} \mathbf{x} +  
\begin{bmatrix}
\beta_1 \\ \vdots \\ \beta_{|\mathcal{A}|}
\end{bmatrix}
\end{equation}

\noindent This function has been used for representing travel times such as in \cite{siri2020progressive}.  \cite{siri2020progressive} labels the $\beta_a$ as the ``free flow travel times'' or i.e. travel times without any interaction with other travelers [\cite{thai2015multi,bertsimas2015data,zhang2017data,zhang2018price,chow2014nonlinear}]. $\phi_a$ is the factor of additional time of having one more unit of flow on the arc.  In this paper, we assume that the free flow travel times are given, and our goal is to estimate $\phi_a$ for all $a \in \mathcal{A}$ (see Section \ref{results_section}).





\item \textbf{Bureau of Public Roads Function:} The Bureau of Public Roads function (BPR) [\cite{BPR_func,branston1976link}] is a common function utilized by transportation researchers when modeling flow along arcs in a network [\cite{thai2018imputing,bertsimas2015data,zhang2017data,zhang2018price}].  From \cite{sheffi_urban_transport}, the BPR function for arc $a$ is:
\begin{equation}
    c_a(\mathbf{x}_a) = t_a^0 \left( 1 + \alpha_a \left(\frac{\mathbf{x}_a}{c'_a} \right)^{\beta} \right).
\end{equation}

\noindent The $t_a^0$ is the free-flow travel time, $c'_a$ is the ``practical capacity'' which we just take as the normal capacity, and $\alpha_a$ \& $\beta$ are parameters which, following \cite{sheffi_urban_transport}, are commonly assumed to be 0.15 and 4 respectively, regardless of the arc.  In contrast, in this paper, we will assume that the $\alpha_a$ parameter is different for each arc and that it is the quantity we estimate with inverse optimization (see Section \ref{results_section}).  We linearize the BPR function using standard techniques [\cite{enme741_textbook,luathep2011global}].

\end{itemize}

\noindent  In our experiments, we compare the protection decisions made under the costs imputed using IO with the protection decisions made when a user has the original parameterization from (\ref{traffic_equilibrium_problem}) and with the protection decisions when a user assumes uniform cost, meaning average $\phi$ for the linear cost function and $0.15$ for the $\alpha$ in the BPR cost function. Section \ref{subsection_metrics} will explain this further.  With the cost functions defined, we discuss the normative model.

\subsection{Normative Model: Two-Stage Stochastic Model}

For the stochastic network protection model portion of the framework, we implement \cite{fan2010solving}'s two-stage network protection model with complementarity constraints with a few changes in the capacity function, the conservation of flow constraints, and the objective function.  We adopt much of the notation from \cite{fan2010solving} and extended definitions for these terms can be found in Appendix B.1: 
\begin{itemize}
    \item $\mathcal{A}$: the set of network arcs, and $m$ as the number of arcs.
    
    \item $\mathcal{N}$: the set of network nodes, and $n$ as the number of nodes.
    
    \item $K$: the number of destinations of flow in the network. 
    
    \item $\mathcal{S}$: the scenario set.
    
    \item $x_a^{k,s}$: the flow on arc $a$ that is destined for the $k$th destination in scenario $s$.  The vector $\mathbf{x}^{k,s} \in \mathbb{R}^m$ denotes the flow on all arcs. (units=thousands of vehicles)
    
    \item $f_{a}^s$: the total flow on arc $a$ in scenario $s$, and $\mathbf{f}^s$ as the vector containing all of the arcs. (units=thousands of vehicles)
    
    \item $u_a$: the decision variable controlling resources used to protect an arc $a$ against a crisis.  
    (units=proportion of necessary resources needed to fully insure the arc) 
    
    \item $W$: the node-link adjacency matrix. 
    
    
    \item $\mathbf{q}^k \in \mathbb{R}^n$: designates the amount of flow originating at each node that is headed to destination $k$.  
    (units = thousands of vehicles)
    
    \item $h_{a}^s(u_a)$: the capacity of an arc $a$ given first stage decision $u_a$ under scenario $s$:
    \begin{equation}
        h_a^s(u_a) = \begin{cases} \text{cap}_a\ \text{ if } a \notin \bar{\mathcal{A}} \\ \text{cap}_a - m_a^s (1-u_a)\ \text{ if } a \in \bar{\mathcal{A}} \end{cases}
    \end{equation}
    
    \noindent with $\text{cap}_a$ representing capacity of the arc without it being affected by a disaster, $m_a^s$ representing the amount of damage done to arc $a$ in scenario $s$ if not protected, and $\bar{\mathcal{A}}$ represents the set of arc vulnerable to the disaster.  Note that $m_a^s$ could be 0 in certain scenarios.  (units = thousands of vehicles)
    
    \item $t_a(\mathbf{f}^s)$ represents the time per vehicle along arc $a$ [\cite{lu2018mean,lu2016bi}] as a function of the flows $\mathbf{f}^s$ in scenario $s$.  We explore multiple different forms for $t_a$, described in Section \ref{subsection_cost_functions}.   
    
    \item $\lambda_i^{k,s}$ as ``the minimum time from node $i$ to destination $k$'' in scenario $s$ [\cite{fan2010solving}]. (units=travel time)
    
    \item $\mathbf{d}^{k,s}$ as the vector of extra variables that acts as a buffer for any flow that cannot be properly apportioned. (units=thousands of vehicles).
    
    \item $p^s$ as the probability of each scenario $s$.
    
    
\end{itemize}

\noindent \cite{fan2010solving}'s model with a modification to the complementarity constraints based on work by \cite{ban2005quasi} and \cite{ban2006general} is thus: 

\begin{subequations}\label{SNPP_formulation}
\begin{equation} \label{objective_func}
    \min \sum\limits_{s\in S} p^s Q^s(\mathbf{u},\mathbf{f}^s)
\end{equation}
\begin{equation} \label{budget_constr}
    \text{s.t. } \mathbf{u} \in \mathcal{D}
\end{equation}
\begin{equation} \label{capacity_constr}
    f^s_a = \sum\limits_{k=1}^K x_a^{k,s} \leq h^s_a(u_a^s),\ a \in \mathcal{A},\ s \in \mathcal{S}
\end{equation}
\begin{equation} \label{complementarity_constr}
    0 \leq x_{ij}^{k,s}\  \bot\  \left(t_a(\mathbf{f}^s) + \lambda_j^{k,s} - \lambda_i^{k,s}\right) \geq 0,\ \forall (i,j) \in \mathcal{A},\ \forall k=1...K,\ \forall s \in \mathcal{S}
\end{equation}
\begin{equation} \label{conservation_constr}
    0 \leq \mathbf{q}^k + \mathbf{d}^{k,s} - W \mathbf{x}^{k,s}\ \bot \ \mathbf{\lambda}^{k,s} \geq 0,\  \forall k=1...K,\ \forall s \in \mathcal{S}
\end{equation}
\end{subequations}

\noindent The $Q^s(\mathbf{u},\mathbf{f}^s)$ function (\ref{objective_func}) in the objective function has the following form: 
\begin{equation*}
    Q^s(\mathbf{u},\mathbf{f}^s) = \langle \mathbf{\psi}, \mathbf{u}^s \rangle + \gamma \langle \mathbf{f}^s, \mathbf{t}(\mathbf{f}^s) \rangle + 10000 \sum\limits_{k=1}^K || \mathbf{d}^{k,s} ||_2^2
\end{equation*}
\noindent The first term denotes the total cost of protection (with $\mathbf{\psi}$ as the dollar amount it costs to protect each arc fully); this term differs from the \cite{fan2010solving} paper which instead uses the cost of repair.  The second term computes the total travel time for all of the flow on each arc, sums these amounts, and then multiplies the sum by $\gamma$, which transforms travel time to financial units [\cite{fan2010solving}], keeping in the same units as the first term.  Note, $\mathbf{t}(\mathbf{f}^s)$ corresponds to the $\mathbf{c}(\mathbf{x})$ function from Section \ref{sec:inverse-opt}.  The third term makes it extremely costly for the model to use any of the buffer that the $\mathbf{d}^{k,s}$ vectors provide for the conservation of flow (\ref{conservation_constr}) constraints.  
Constraints (\ref{budget_constr}) represent the ``budgetary and technological restrictions'' [\cite{fan2010solving}].  We specified this further as just budgetary constraints of the form:
\begin{equation}\label{true_budget_constraint}
    \sum\limits_{a \in \mathcal{A}} u_a \leq I
\end{equation}

\noindent with $I$ representing the number of arcs we can afford to fully protect.  However, because $u_a$ are continuous variables, we can protect more than $I$ number of arcs partially because we are treating $u_a$ as proportions.  The capacity constraints (\ref{capacity_constr}) have an $s$ dependence for the $h_a^s$ functions because the $m_a^s\ \forall a \in \bar{\mathcal{A}}$ are scenario-dependent.  The constraints in (\ref{complementarity_constr}) encapsulate the idea that there should be no flow on the arc $a$ on its way to destination $k$ in scenario $s$ unless that arc is part of the minimal travel time route to destination $k$.  The complementarity constraints in (\ref{conservation_constr}) include the conservation of flow constraints that ensure flow begins and ends at the appropriate places in the network.  The $\mathbf{d}^{k,s}$ vectors are buffers in case some of this flow does not fulfill the conservation of flow constraints; \cite{fan2010solving} define them as variables to ensure a feasible solution.   

We specify some of the parameters for the model that will not change over the course of the paper:

\begin{itemize}
    \item The $\text{cap}_a$ value is set to 8 for all arcs $a$.
    
    \item The $m_a^s$ value (amount of damage) is set to 8.
    
    \item In the objective function, we set $\gamma = 1$ (following \cite{fan2010solving}) and set the $\psi$ vector to 1 because we do not want cost to be prohibitive.
    
\end{itemize}

\subsubsection{Big M Method for Complementarity Constraints and Progressive Hedging Algorithm for Solving Two-Stage Problem}

\cite{fan2010solving} note in their paper that this problem is difficult to solve because of (1) the complementarity constraints and (2) the stochastic elements.  In order to handle the complementarity constraints, we use the disjunctive constraint/big M method approach [\cite{fortuny1981representation,hart2017mathematical}].  As an example, we take the complementarity condition from (\ref{complementarity_constr}) and produce a series of constraints:
\begin{subequations}
\begin{equation}
    x_{ij}^{k,s} \geq 0
\end{equation}
\begin{equation}
    t_a(\mathbf{f}^s) + \lambda_j^{k,s} - \lambda_i^{k,s} \geq 0
\end{equation}
\begin{equation}\label{comp_constraint_1}
    x_{ij}^{k,s} \leq M_{ij}^{k,s} (b_{ij}^{k,s})
\end{equation}
\begin{equation}\label{comp_constraint_2}
    \left(t_a(\mathbf{f}^s) + \lambda_j^{k,s} - \lambda_i^{k,s} \right) \leq M_{ij}^{k,s} (1-b_{ij}^{k,s})
\end{equation}
\end{subequations}

\noindent The $b_{ij}^{k,s}$ is a binary variable, and $M_{ij}^{k,s}$ is a sufficiently large number, which forces at least one of the two terms in (\ref{comp_constraint_1}) or (\ref{comp_constraint_2}) to be 0.  We repeat the same procedure for the complementarity constraints in (\ref{conservation_constr}).  See Appendix B.2 for information on calculating the $M_{ij}^{k,s}$ values.  To handle the stochasticity of this problem, we follow \cite{fan2010solving} by employing the progressive hedging (PH) algorithm.  Proposed by \cite{rockafellar1991scenarios}, the PH algorithm at its most basic level solves scenario subproblems created by the random variable(s) involved in the original problem using an approach in which there is a penalty term that encourages first-stage variables to tend toward the ``aggregate'' solution [\cite{rockafellar1991scenarios}] that is computed after each iteration of the algorithm.  See the following references for more information about using the algorithm and setting its parameters: [\cite{watson2011progressive,ryan2013toward,carpentier2013long,veliz2015stochastic,fan2010solving,crainic2011progressive,gonccalves2012applying,palsson1994stochastic,gul2015progressive,hvattum2009using,lamghari2016progressive,mulvey1991applying}].
We use the implementation of the PH algorithm found in the \url{pysp} extension [\cite{watson2012pysp}] of the \url{pyomo} package [\cite{hart2011pyomo,hart2017mathematical}] in Python.  We use \url{gurobi} for the mixed-integer quadratic programming sub-problems arising as part of the PH algorithm.

\section{Experimental Study}\label{experimental-study}

In this section, we define our experimental setup and the metrics by which we will evaluate the experiments.  Most of our experimental results demonstrate inverse optimization enables users to recover comparable protection decisions as the original cost protection decisions, and there is a difference between protection decisions made under uniform cost parameters and the original or IO parameterizations.

\subsection{Experimental Setup}\label{simulation_structure_section}

First, we consider two networks on which to test our hybrid framework: a 4x4 grid in Figure \ref{fig:network-grid} 
and the Nguyen \& Dupuis network [\cite{nguyen1984efficient}] in Figure \ref{fig:network-ND}, both of which we make bidirectional.

\begin{figure}[H]
\begin{subfigure}{0.5\textwidth}
\centering
\includegraphics[height=0.2\textheight,keepaspectratio]{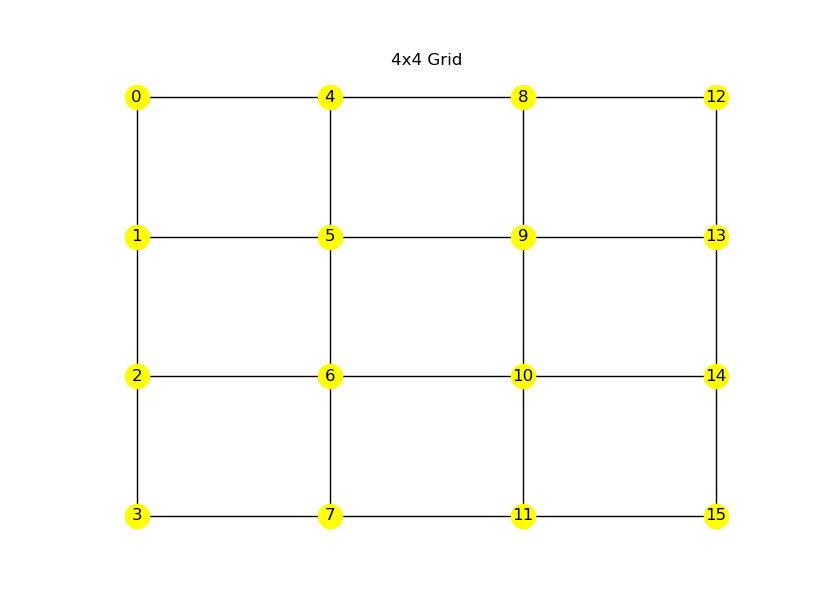}
\caption{4x4 Directed Grid Network, 16 Nodes \& 48 Arcs}\label{fig:network-grid}
\end{subfigure}
\begin{subfigure}{0.5\textwidth}
\centering
\includegraphics[height=0.2\textheight,keepaspectratio]{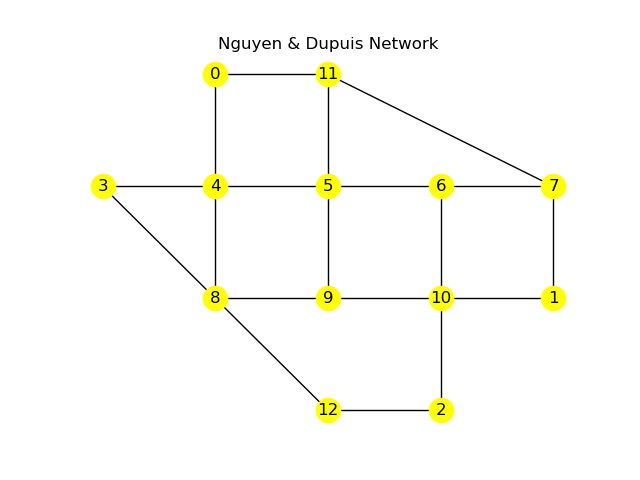}
\caption{Nguyen \& Dupuis Network, 13 Nodes \& 38 Arcs}\label{fig:network-ND}
\end{subfigure}
\caption{Illustrative Road Networks Utilized for Experiments}
\label{fig:networks}
\end{figure}

\vspace{1em}

\noindent We generated data (observations of flow $\hat{\mathbf{x}}^j$ for all $j=1,...,J$), as discussed in Section 2.1, using the forward problem in the form of the complementarity model (\ref{traffic_equilibrium_problem}); we solve (\ref{traffic_equilibrium_problem}) using PATH [\cite{dirkse1995path,ferris2000complementarity}] in GAMS. The set $\Lambda$ represents the origin-destination pairs utilized for each run of the complementarity model.  For this paper, $\Lambda$ is the set of all different origin-destination (OD) pairs for each network for the data generation process, one pair for each run of the complementarity model.  In more complicated versions, $\Lambda$ would consist of multiple different OD pairs per run of the complementarity model. 
The algorithm below illustrates generating the $\hat{\mathbf{x}}^j$ data for $j=1,...,J$ given a set of configurations $\Lambda$ with $|\Lambda| = J$. 

\vspace{1em}

\begin{algorithm}[H]
\KwData{The set $\Lambda$ of configurations}
\For{$j$ = 1:$|\Lambda|$}{
Build the traffic equilibrium model (\ref{traffic_equilibrium_problem}) with the $j$th configuration of OD pair(s)\\
Solve the traffic equilibrium model (\ref{traffic_equilibrium_problem}) with PATH in GAMS\\
Store optimal $\hat{\mathbf{x}}^j$ \\
}
\caption{Generating the Data}
\end{algorithm}

\vspace{1em}

\noindent The generated data $\hat{\mathbf{x}}^j, j=1...J$ is used as input into the inverse optimization model to determine the parameters $\theta$ for the cost function.  We then carry through with the rest of the hybrid framework described at the beginning of Section \ref{framework-explanation} to obtain the protection decisions.

\subsection{Metrics}\label{subsection_metrics}

In order to evaluate our hybrid framework, we must solve the stochastic network protection problem three times for each set of generated data because we must compare the protection decisions under the original cost parameters, the inverse cost parameters, and the assumption of uniform cost parameters: 

\begin{itemize}

\item We define the ``IO information protection decisions'', denoted $\mathbf{u}$, as the protection decisions the two-stage model would make based on the cost vector obtained using the inverse optimization algorithm.

\item We define the ``original information protection decisions'', denoted $\hat{\mathbf{u}}$, as the protection decisions that the two-stage model would make if it were directly given the ``original'' cost structure that (\ref{traffic_equilibrium_problem}) generated the data $\hat{\textbf{x}}^j, j=1...J$.  The goal of the framework is for the inverse optimization algorithm to be able to provide a cost estimate that will result in the same/comparable protection decisions as under the original cost structure.  We use the Original-IO metric below to evaluate the ``comparability'' of the protection decisions, which refers to the metric being close to 0.

\item We define the ``uniform information protection decisions'', denoted $\bar{\mathbf{u}}$, as those decisions that the two-stage model would make if it were given a uniform cost structure for the network.  If the original information decisions differ significantly from the uniform information decisions, then this provides evidence that knowing the cost structure of the network is important.

\end{itemize}

\noindent Our performance metrics capture the difference between the protection decisions made under different costs:
\begin{itemize}

    \item \textbf{Original-IO (O-IO)}: $||\hat{\mathbf{u}}-\mathbf{u} ||_2$
    
    \item \textbf{Uniform-IO (U-IO)}: $||\bar{\mathbf{u}} - \mathbf{u} ||_2$
    
    \item \textbf{Uniform-Original (U-O)}: $|| \bar{\mathbf{u}} - \hat{\mathbf{u}} ||_2$.
    
\end{itemize}

\subsection{Results}\label{results_section}

In this section, the hybrid inverse optimization-stochastic programming framework displayed at the beginning of Section \ref{framework-explanation} is tested through a series of experiments. For each trial of each experiment, Algorithm 1 is used to generate $\hat{\mathbf{x}}^j$ for $j=1...J$.  Next, for each trial of each experiment, the hybrid framework is used to estimate a set of parameters for the current cost function and to find the protection decisions under that parameterization given a set budget.

In the experiments, two components are varied: the type of graph (4x4 grid vs. Nguyen \& Dupuis) and the type of cost function (linear vs. BPR functions).  For all of the experiments, the following remain the same:
\begin{itemize}
    \item The MATLAB built-in function \url{unifrnd} is used to create the random original costs for the data generation part of each trial of each experiment.
    
    \item There are 10 trials for each experiment, each with a different original (and random) cost parameterization.
    
    \item The $k$ for $\mathbf{q}^k$ is equal to 2.  For the 4x4 Grid, eight units of flow go from node 0 to node 15 and from 15 to node 0 and, for the Nguyen \& Dupuis network, eight units of flow go from node 0 to node 2 and from node 2 to node 0.  See Figure \ref{fig:networks} for the node references.
    
    \item The set of scenarios correspond to each pair of arcs indicated in Figures \ref{fig:uniform_grid_scenarios} and \ref{fig:uniform_ND_scenarios} in the next subsection.  Each set has an equal chance of failing.
    
    \item The budget for the constraint (\ref{true_budget_constraint}) is set to 6.  It could be modified in future work.
    
    \item For each run of the stochastic network protection problem (SNPP), $\epsilon=0.01$ or $\epsilon=0.001$ for the $g^{(k)}$ PH error metric [\cite{watson2011progressive}]; we utilize the default error metric outlined in the \url{pysp} documentation in \cite{hart2017pyomo}.  The maximum number of iterations is 300.  Therefore, the progressive hedging algorithm stops if it reaches the tolerance or if it reaches the 300 iterations, whichever occurs first.  Two runs of each experiment are done, one under $\epsilon=0.01$ and one under $\epsilon=0.001$.
\end{itemize}

\subsubsection{Detailed Experiment Descriptions}\label{subsec:experiment_descriptions}

Table \ref{table:experiments} presents the experiments for the $4\times4$ Grid network presented in Figure \ref{fig:network-grid} and the Nguyen \& Dupuis network presented in Figure \ref{fig:network-ND}.  The labels of Table \ref{table:experiments} correspond to the following meanings:
\begin{itemize}
    \item Network: The network used  
    \item Cost Function: The transportation function used 
    \item Parameters: Distributions from which the original cost function coefficients are drawn
    \item \# Scenarios: References the number of arc pairs that are vulnerable to destruction
    \item $\rho$ Used: Indicates a parameter value in the PH Algorithm 
    \item Number of Cores: Refers to the number of computer cores utilized for the Experiments
\end{itemize}

\noindent The protection decisions made under the original, uniform, and IO costs --- hereafter referred to as $\hat{\theta}, \bar{\theta},$ and $\theta$ respectively --- are examined for each experiment in the next subsection. 

\begin{table}[H]
    \centering
    \begin{tabular}{c|c|c|c|c}
        \textbf{Experiment \#} & I & II & III & IV \\ \hline 
        \textbf{Network} & 4x4 Grid & 4x4 Grid & N \& D & N \& D \\ \hline 
        \textbf{Cost Function} & Linear & BPR & Linear & BPR \\ 
        \hline
        \textbf{Parameters} & \makecell{$\phi_a \sim U[2,10]$ \\ $\beta_a \sim U[2,10]$} & \makecell{$\alpha_a \sim U[0.1,0.2]$\\$t^0_a \sim U[2,10]$ \\ $c'_a = 8$} & \makecell{$\phi_a \sim U[2,10]$ \\ $\beta_a \sim U[2,10]$} & \makecell{$\alpha_a \sim U[0.1,0.2]$\\$t^0_a \sim U[2,10]$ \\ $c'_a = 8$} \\ 
        \hline
        \textbf{\# Scenarios} & 12 & 12 & 9 & 9 \\ 
        \hline
        \textbf{$\rho$ Used} & 5 & 5 & 5 & 5\\ \hline
        \textbf{Number of Cores} & 8 & 8 & 8 & 8 \\
        \hline
        
    \end{tabular}
    \caption{Experiment Descriptions}
    \label{table:experiments}
\end{table}

\noindent \textbf{Experiment I:} In the first experiment, the linear cost function along with the 4x4 grid are utilized.  The linear cost function is $\phi_a x_a + \beta_a$ for each arc $a$.  As mentioned at the beginning of Section \ref{results_section}, the scenarios are chosen such that every other pair of arcs are vulnerable to complete destruction. This can be seen in Figure \ref{fig:uniform_grid_scenarios} through the placement of the triangles with ``lightening bolts,'' indicating the arc pairs at risk.  Each arc pair is given a 1/12 chance of failing.

\begin{figure}[H]
\begin{center}
\includegraphics[height=0.22\textheight,keepaspectratio]{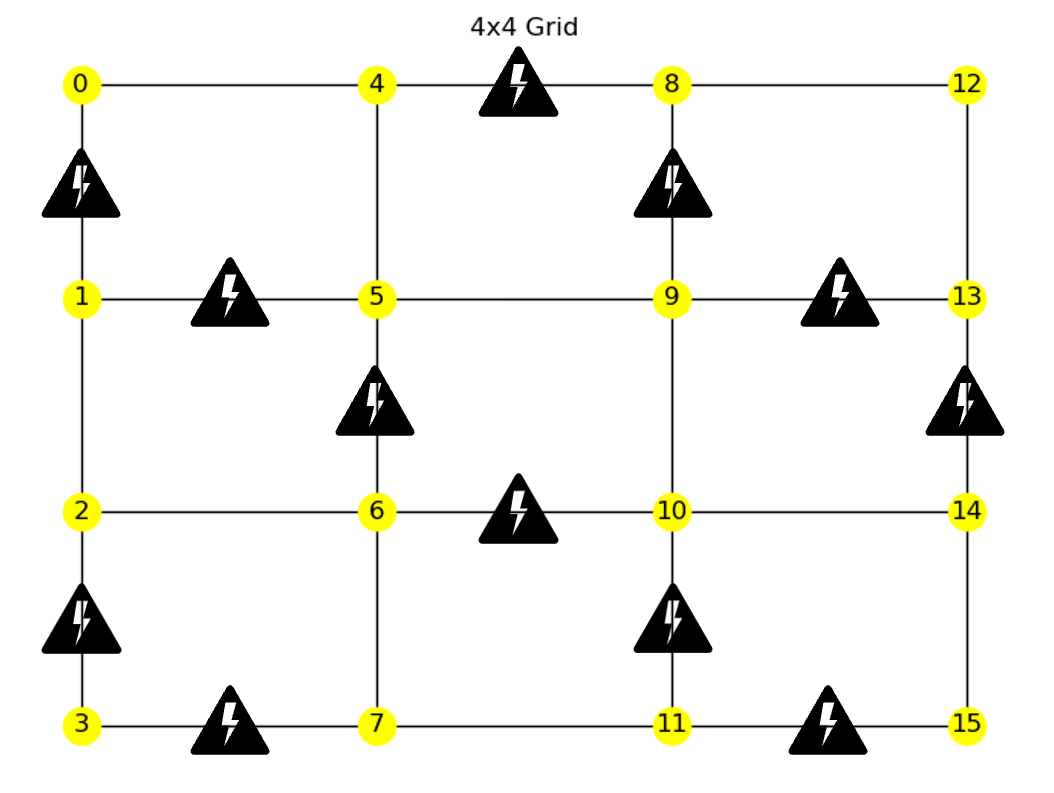}
\caption{Scenarios for Experiments I \& II}
\label{fig:uniform_grid_scenarios}
\end{center}
\end{figure}

Furthermore, for the experiments involving the linear cost function, only the $\phi_a$ parameters are estimated.  $\beta_a$, the free flow travel time for arc $a$, is assumed to be known in the majority of the papers cited in the literature review on estimating cost functions using inverse optimization [\cite{thai2015multi,thai2018imputing,zhang2017data,chow2014nonlinear,bertsimas2015data,zhang2018price}]. 
Consequently, for the $\hat{\theta}$, $\bar{\theta}$, and $\theta$ stochastic protection models involving linear cost, the $\beta_a$ terms are the same across all of them.  For the $\hat{\theta}$ protection model, the $\phi_a$ terms are the original cost values generated by the \url{unifrnd} MATLAB function.  For the $\theta$ protection model, the $\phi_a$ terms come from the inverse optimization model (\ref{basic_bertsimas_formulation}).  Finally, for the $\bar{\theta}$ protection model, the $\phi_a = 6$ for all $a$, as indicated in Table \ref{table:experiments}.

\vspace{1em}

\noindent \textbf{Experiment II:} In Experiment II, the inverse optimization algorithm computes the $\alpha_a$ parameters for each arc $a$ for the BPR cost function.  \cite{wong2016network} support having different $\alpha_a$ parameters across the network because they vary the $\alpha$ based on the structure of the network involved.  \cite{lu2016estimating} create their BPR function such that the $\alpha$ parameter value differs for each type of vehicle in their simulation, thus again showing that the $\alpha$ parameter can be different than the standard uniform 0.15 noted in Section \ref{subsection_cost_functions}.  The IO model is given the randomly chosen $t^0_a$ parameters and the capacity levels (all set to 8), and it is asked to estimate the $\alpha_a$ values.  The uniform parameter value that is chosen for the BPR experiments is 0.15 because that is the traditionally chosen parameter value in most models as noted in Section \ref{subsection_cost_functions} with the source \cite{sheffi_urban_transport}.  The scenario set up is the same as in Experiment I (see Figure \ref{fig:uniform_grid_scenarios}).

\vspace{1em}

\noindent \textbf{Experiment III:} Experiment III utilizes the linear cost function for the Nguyen \& Dupuis network. Figure \ref{fig:uniform_ND_scenarios} illustrates the arcs that have a chance of failing, which are again chosen such that every other pair of arcs are vulnerable to complete destruction. There are 9 pairs of arcs indicated, which means each pair has a 1/9 chance of completely failing.  

\begin{figure}[H]
\begin{center}
\includegraphics[height=0.22\textheight,keepaspectratio]{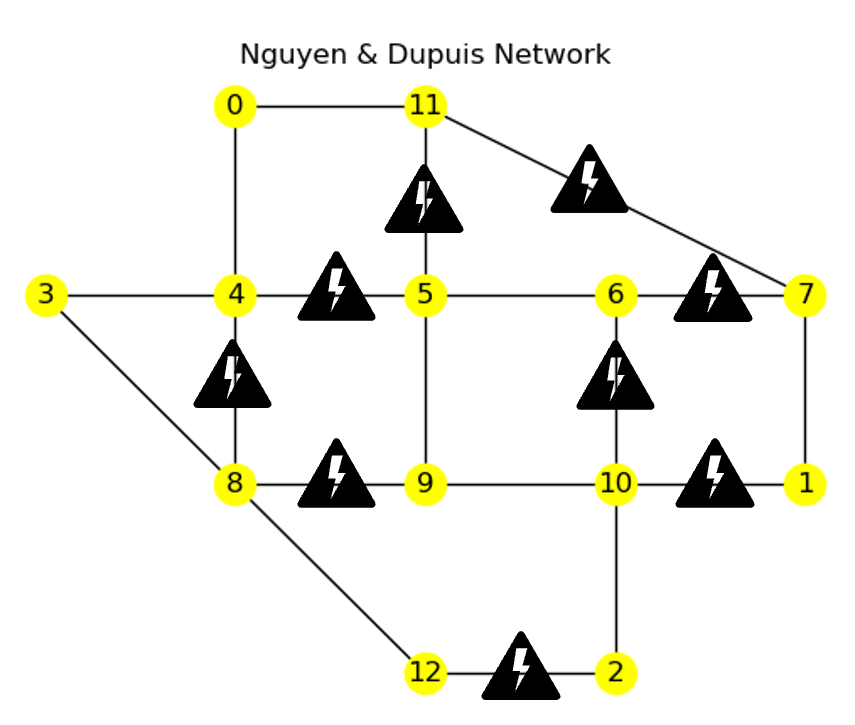}
\caption{Scenarios for Experiments III \& IV}
\label{fig:uniform_ND_scenarios}
\end{center}
\end{figure}

\noindent The notes about the linear cost function discussed in Experiment I hold for this experiment; only the network has changed along with the $\mathbf{q}^k$ for the SNPP. See the beginning of Section \ref{results_section} for the note about $\mathbf{q}^k$. 

\vspace{1em}

\noindent \textbf{Experiment IV:} Experiment IV employs the BPR cost function along with the Nguyen \& Dupuis network.  The scenario pattern is the same as in Experiment III (see Figure \ref{fig:uniform_ND_scenarios}), and the description of Experiment II for the BPR function holds for this experiment as well.  

\subsubsection{Conclusions from the Experiments}\label{subsec:conclusions_from_experiments}

The results of the experiments are evaluated with respect to two hypotheses regarding confidence intervals of the means and to a direct comparison of the means of the O-IO, U-IO, and U-O metrics defined in Section \ref{subsection_metrics}.\footnote{See Appendix C.1 for the flow error metrics for the IO $\alpha$ values because, as can be seen from Figure \ref{fig:experiment_2_metrics}, the IO $\alpha$ values are different from the original $\alpha$ values.} See Appendix C.2 for information on the medians \& minimum/maximums of the data and see Appendix C.3 for the run times of the experiments.  

The first hypothesis states the O-IO confidence intervals for the experiments will include 0 because the inverse optimization framework can recover comparable protection decisions to the $\hat{\theta}$ protection decisions.  Looking at Tables \ref{table:means_CI_99_epsilon_0_01}-\ref{table:means_CI_99_epsilon_0_001}, all of the confidence intervals include 0.  Therefore, there is evidence in favor of the hypothesis.

The second hypothesis states that the U-IO and U-O metric confidence intervals will not include 0 because having either cost parameters that are learned from IO or the original parameters makes a difference in protection decisions when compared to the case of uniform cost parameters.  Tables \ref{table:means_CI_99_epsilon_0_01}-\ref{table:means_CI_99_epsilon_0_001} indicate that Experiments I-III present evidence in favor of the hypothesis, but Experiment IV falls short since the confidence intervals for U-IO and U-O both include 0.

Experiments I-III demonstrate that using $\bar{\theta}$ (uniform cost) leads to different protection decisions than the $\hat{\theta}$ or $\theta$ costs.  Comparing the means as a percentage of the total budget in Tables \ref{table:means_CI_99_epsilon_0_01} and \ref{table:means_CI_99_epsilon_0_001}, we see that the O-IO metric as a percentage of the budget is small, while the U-IO and the U-O metrics as a percentage of the budget are many times greater.  The small values of the O-IO metric as a percentage of the budget indicate that IO can be used to recover parameters for the SNPP model, while the comparatively large values of the U-IO and U-O metrics indicate that the uniform protection decisions are quite different from the IO and original protection decisions, confirming the value of IO in recovering the network parameters.  Boxplots in Figures \ref{fig:experiment_1_metrics} and \ref{fig:experiment_2_metrics} (along with the boxplots for Experiment III in Appendix C.2) tell the same story.

\begin{table}[H]
\centering
\begin{footnotesize}
\begin{tabular}{c|cc|cc|cc|cc}
     & \multicolumn{2}{c}{\textbf{Experiment I}} & \multicolumn{2}{c}{\textbf{Experiment II}} & \multicolumn{2}{c}{\textbf{Experiment III}} & \multicolumn{2}{c}{\textbf{Experiment IV}} \\ \cline{2-9}
     & Mean & CI & Mean & CI & Mean & CI & Mean & CI\\ \hline
     \textbf{O-IO}& \makecell{0.0152\\0.25\%} &(-0.0101, 0.0405)& \makecell{0.0015\\0.03\%} &(-0.0007, 0.0038)& \makecell{0.0263\\0.44\%} &(-0.0037, 0.0564)& \makecell{0.0041\\0.07\%} &(-0.0017, 0.0099)\\ \hline \textbf{U-IO}& \makecell{0.3473\\5.79\%} &(0.2796, 0.4149)& \makecell{0.0564\\0.94\%} &(0.0115, 0.1013)& \makecell{0.2668\\4.45\%} &(0.2077, 0.3259)& \makecell{0.0261\\0.43\%} &(-0.028, 0.0801)\\ \hline \textbf{U-O}& \makecell{0.3466\\5.78\%} &(0.2778, 0.4153)& \makecell{0.0559\\0.93\%} &(0.0108, 0.1009)& \makecell{0.264\\4.4\%} &(0.1982, 0.3299)& \makecell{0.0264\\0.44\%} &(-0.0285, 0.0812)
\end{tabular}
\end{footnotesize}
\caption{Means, Means as Percentage of $I=6$ Budget, and 99\% Confidence Intervals (CI) for Experiments, $\epsilon=0.01$} \label{table:means_CI_99_epsilon_0_01}
\end{table}

\begin{table}[H]
\centering
\begin{footnotesize}
\begin{tabular}{c|cc|cc|cc|cc}
     & \multicolumn{2}{c}{\textbf{Experiment I}} & \multicolumn{2}{c}{\textbf{Experiment II}} & \multicolumn{2}{c}{\textbf{Experiment III}} & \multicolumn{2}{c}{\textbf{Experiment IV}} \\ \cline{2-9}
     & Mean & CI & Mean & CI & Mean & CI & Mean & CI\\ \hline
     \textbf{O-IO}& \makecell{0.0148\\0.25\%} &(-0.0107, 0.0402)& \makecell{0.0013\\0.02\%} &(-0.0002, 0.0027)& \makecell{0.0304\\0.51\%} &(-0.0147, 0.0756)& \makecell{0.0117\\0.19\%} &(-0.0168, 0.0402)\\ \hline \textbf{U-IO}& \makecell{0.3313\\5.52\%} &(0.2566, 0.4059)& \makecell{0.0638\\1.06\%} &(0.01, 0.1176)& \makecell{0.2643\\4.41\%} &(0.1958, 0.3329)& \makecell{0.0197\\0.33\%} &(-0.016, 0.0554)\\ \hline \textbf{U-O}& \makecell{0.3308\\5.51\%} &(0.2557, 0.4058)& \makecell{0.0634\\1.06\%} &(0.0095, 0.1173)& \makecell{0.2594\\4.32\%} &(0.187, 0.3317)& \makecell{0.0257\\0.43\%} &(-0.0158, 0.0673) \\
\end{tabular}
\end{footnotesize}
\caption{Means, Means as Percentage of $I=6$ Budget, and 99\% Confidence Intervals (CI) for Experiments, $\epsilon=0.001$} \label{table:means_CI_99_epsilon_0_001}
\end{table}

\begin{figure}[H]
\begin{subfigure}{0.5\textwidth}
\centering
\includegraphics[height=0.2\textheight,keepaspectratio]{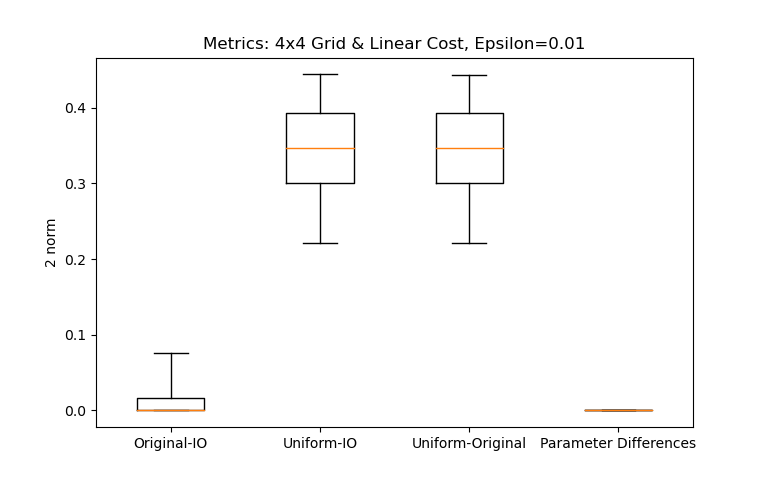}
\caption{Experiment I Results for $\epsilon=0.01$}\label{fig:experiment_1_0_01}
\end{subfigure}
\begin{subfigure}{0.5\textwidth}
\centering
\includegraphics[height=0.2\textheight,keepaspectratio]{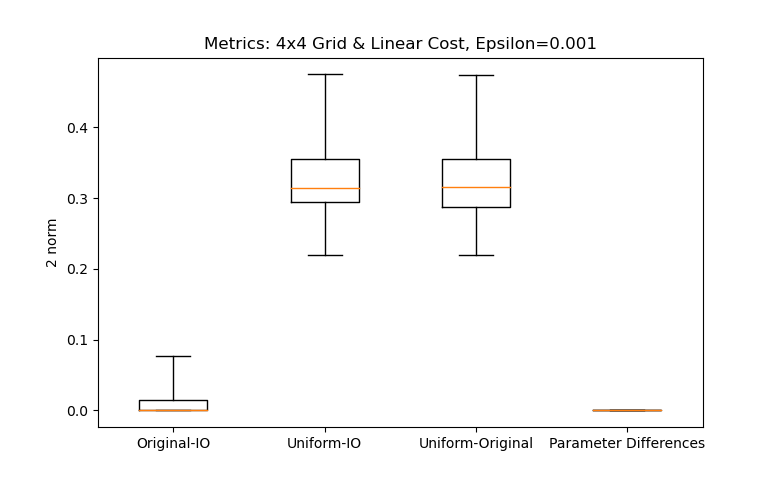}
\caption{Experiment I Results for $\epsilon=0.001$}\label{fig:experiment_1_0_001}
\end{subfigure}
\caption{Experiment I Results: 4x4 Grid Network with Linear Cost.  The parameter differences refer to the $\phi$ differences.}
\label{fig:experiment_1_metrics}
\end{figure}

\begin{figure}[H]
\begin{subfigure}{0.6\textwidth}
\centering
\includegraphics[height=0.2\textheight,keepaspectratio]{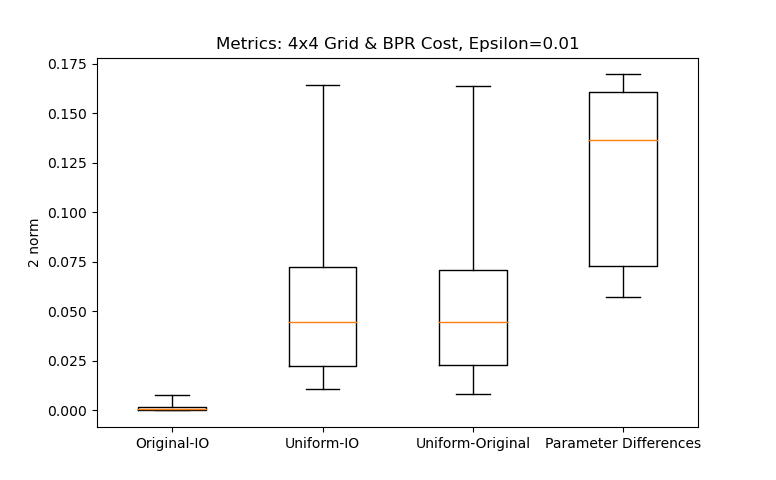}
\caption{Experiment II Results for $\epsilon=0.01$}\label{fig:experiment_2_0_01}
\end{subfigure}
\begin{subfigure}{0.4\textwidth}
\centering
\includegraphics[height=0.2\textheight,keepaspectratio]{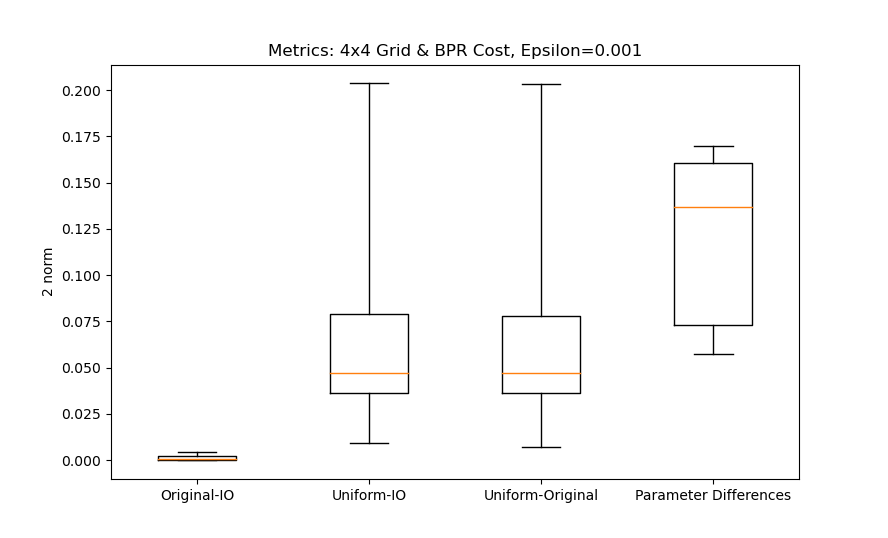}
\caption{Experiment II Results for $\epsilon=0.001$}\label{fig:experiment_2_0_001}
\end{subfigure}
\caption{Experiment II Results: 4x4 Grid with BPR Function.  Parameter differences here refers to the $\alpha$ differences.}
\label{fig:experiment_2_metrics}
\end{figure}

\section{Conclusions \& Future Work}\label{conclusions}

In this paper, we have demonstrated that inverse optimization can be used as a tool to make better protection decisions in multi-stage stochastic programs for disaster relief.  Through experiments with two different networks and two different cost functions, we demonstrate that IO can be used to recover network parameters that produce similar protection decisions as the original parameters that were used to generate the data in Algorithm 1. For most of the experiments, we also demonstrate that the protection decisions are different when we have either cost parameters learned from IO or the original cost parameters compared to the protection decisions that would have been made under the assumption of uniform cost.  Finally, we show there is a difference between the means of the O-IO versus U-IO and U-O metrics for Experiments I-III.  These results suggest that inverse optimization can be used as a data analysis approach in a DSS and as a way to estimate cost parameters in multi-stage stochastic programs for disaster management.

With regard to future work, estimating costs such that they are a function of the disaster would be something worth pursuing; indeed, it may be possible to incorporate risk metrics such as those proposed by \cite{cantillo2019assessing} and \cite{guo2017seismic}.  In addition, expanding the experiments such that there is interaction between OD pairs in both the data set for the IO mathematical program and in the flow patterns for the SNPP would enrich the analysis.  Furthermore, obtaining real data on scenarios and on traffic patterns would allow us to take these simulated results and apply them to the real world.

\section{Acknowledgements}

Allen was partially funded by a Graduate Fellowship in STEM Diversity while completing this research (formerly known as a National Physical Science Consortium Fellowship).  Allen was also supported by a Flagship and Dean's Fellowships from the University of Maryland, College Park and has worked for Johns Hopkins University Applied Physics Lab in the summers.  Terekhov and Gabriel have no funding sources or conflicts of interest to report.

We would like to thank Dr. David Woodruff of University of California, Davis for answering questions regarding the \url{pysp} package.


\bibliographystyle{apacite}
\bibliography{interactapasample}

\begin{thebibliography}{}

\bibitem [\protect \citeauthoryear {%
Abpeykar%
\ \BBA {} Ghatee%
}{%
Abpeykar%
\ \BBA {} Ghatee%
}{%
{\protect \APACyear {2014}}%
}]{%
abpeykar2014supervised}
\APACinsertmetastar {%
abpeykar2014supervised}%
\begin{APACrefauthors}%
Abpeykar, S.%
\BCBT {}\ \BBA {} Ghatee, M.%
\end{APACrefauthors}%
\unskip\
\newblock
\APACrefYearMonthDay{2014}{}{}.
\newblock
{\BBOQ}\APACrefatitle {Supervised and unsupervised learning DSS for incident
  management in intelligent tunnel: A case study in Tehran Niayesh tunnel}
  {Supervised and unsupervised learning dss for incident management in
  intelligent tunnel: A case study in tehran niayesh tunnel}.{\BBCQ}
\newblock
\APACjournalVolNumPages{Tunnelling and Underground Space
  Technology}{42}{}{293--306}.
\PrintBackRefs{\CurrentBib}

\bibitem [\protect \citeauthoryear {%
Ahuja%
\ \BBA {} Orlin%
}{%
Ahuja%
\ \BBA {} Orlin%
}{%
{\protect \APACyear {2001}}%
}]{%
ahuja2001inverse}
\APACinsertmetastar {%
ahuja2001inverse}%
\begin{APACrefauthors}%
Ahuja, R\BPBI K.%
\BCBT {}\ \BBA {} Orlin, J\BPBI B.%
\end{APACrefauthors}%
\unskip\
\newblock
\APACrefYearMonthDay{2001}{}{}.
\newblock
{\BBOQ}\APACrefatitle {Inverse optimization} {Inverse optimization}.{\BBCQ}
\newblock
\APACjournalVolNumPages{Operations Research}{49}{5}{771--783}.
\PrintBackRefs{\CurrentBib}

\bibitem [\protect \citeauthoryear {%
Allen%
, Dickerson%
\BCBL {}\ \BBA {} Gabriel%
}{%
Allen%
\ \protect \BOthers {.}}{%
{\protect \APACyear {2021}}%
}]{%
allen2021using}
\APACinsertmetastar {%
allen2021using}%
\begin{APACrefauthors}%
Allen, S.%
, Dickerson, J\BPBI P.%
\BCBL {}\ \BBA {} Gabriel, S\BPBI A.%
\end{APACrefauthors}%
\unskip\
\newblock
\APACrefYearMonthDay{2021}{}{}.
\newblock
{\BBOQ}\APACrefatitle {Using Inverse Optimization to Learn Cost Functions in
  Generalized Nash Games} {Using inverse optimization to learn cost functions
  in generalized nash games}.{\BBCQ}
\newblock
\APACjournalVolNumPages{arXiv preprint arXiv:2102.12415}{}{}{}.
\PrintBackRefs{\CurrentBib}

\bibitem [\protect \citeauthoryear {%
Alvear%
, Abreu%
, Cuesta%
\BCBL {}\ \BBA {} Alonso%
}{%
Alvear%
\ \protect \BOthers {.}}{%
{\protect \APACyear {2013}}%
}]{%
alvear2013decision}
\APACinsertmetastar {%
alvear2013decision}%
\begin{APACrefauthors}%
Alvear, D.%
, Abreu, O.%
, Cuesta, A.%
\BCBL {}\ \BBA {} Alonso, V.%
\end{APACrefauthors}%
\unskip\
\newblock
\APACrefYearMonthDay{2013}{}{}.
\newblock
{\BBOQ}\APACrefatitle {Decision support system for emergency management: Road
  tunnels} {Decision support system for emergency management: Road
  tunnels}.{\BBCQ}
\newblock
\APACjournalVolNumPages{Tunnelling and underground space
  technology}{34}{}{13--21}.
\PrintBackRefs{\CurrentBib}

\bibitem [\protect \citeauthoryear {%
Asadabadi%
\ \BBA {} Miller-Hooks%
}{%
Asadabadi%
\ \BBA {} Miller-Hooks%
}{%
{\protect \APACyear {2017}}%
}]{%
asadabadi2017optimal}
\APACinsertmetastar {%
asadabadi2017optimal}%
\begin{APACrefauthors}%
Asadabadi, A.%
\BCBT {}\ \BBA {} Miller-Hooks, E.%
\end{APACrefauthors}%
\unskip\
\newblock
\APACrefYearMonthDay{2017}{}{}.
\newblock
{\BBOQ}\APACrefatitle {Optimal transportation and shoreline infrastructure
  investment planning under a stochastic climate future} {Optimal
  transportation and shoreline infrastructure investment planning under a
  stochastic climate future}.{\BBCQ}
\newblock
\APACjournalVolNumPages{Transportation Research Part B:
  Methodological}{100}{}{156--174}.
\PrintBackRefs{\CurrentBib}

\bibitem [\protect \citeauthoryear {%
J\BPBI X.~Ban%
, Liu%
, Ferris%
\BCBL {}\ \BBA {} Ran%
}{%
J\BPBI X.~Ban%
\ \protect \BOthers {.}}{%
{\protect \APACyear {2006}}%
}]{%
ban2006general}
\APACinsertmetastar {%
ban2006general}%
\begin{APACrefauthors}%
Ban, J\BPBI X.%
, Liu, H\BPBI X.%
, Ferris, M\BPBI C.%
\BCBL {}\ \BBA {} Ran, B.%
\end{APACrefauthors}%
\unskip\
\newblock
\APACrefYearMonthDay{2006}{}{}.
\newblock
{\BBOQ}\APACrefatitle {A general MPCC model and its solution algorithm for
  continuous network design problem} {A general mpcc model and its solution
  algorithm for continuous network design problem}.{\BBCQ}
\newblock
\APACjournalVolNumPages{Mathematical and Computer
  Modelling}{43}{5-6}{493--505}.
\PrintBackRefs{\CurrentBib}

\bibitem [\protect \citeauthoryear {%
X\BPBI J.~Ban%
}{%
X\BPBI J.~Ban%
}{%
{\protect \APACyear {2005}}%
}]{%
ban2005quasi}
\APACinsertmetastar {%
ban2005quasi}%
\begin{APACrefauthors}%
Ban, X\BPBI J.%
\end{APACrefauthors}%
\unskip\
\newblock
\APACrefYear{2005}.
\newblock
\APACrefbtitle {Quasi-variational inequality formulations and solution
  approaches for dynamic user equilibria} {Quasi-variational inequality
  formulations and solution approaches for dynamic user equilibria}.
\newblock
\APACaddressPublisher{}{The University of Wisconsin-Madison}.
\PrintBackRefs{\CurrentBib}

\bibitem [\protect \citeauthoryear {%
Barbarosoglu%
\ \BBA {} Arda%
}{%
Barbarosoglu%
\ \BBA {} Arda%
}{%
{\protect \APACyear {2004}}%
}]{%
barbarosoglu2004two}
\APACinsertmetastar {%
barbarosoglu2004two}%
\begin{APACrefauthors}%
Barbarosoglu, G.%
\BCBT {}\ \BBA {} Arda, Y.%
\end{APACrefauthors}%
\unskip\
\newblock
\APACrefYearMonthDay{2004}{}{}.
\newblock
{\BBOQ}\APACrefatitle {A two-stage stochastic programming framework for
  transportation planning in disaster response} {A two-stage stochastic
  programming framework for transportation planning in disaster
  response}.{\BBCQ}
\newblock
\APACjournalVolNumPages{Journal of the operational research
  society}{55}{1}{43--53}.
\PrintBackRefs{\CurrentBib}

\bibitem [\protect \citeauthoryear {%
Bertsimas%
, Gupta%
\BCBL {}\ \BBA {} Paschalidis%
}{%
Bertsimas%
\ \protect \BOthers {.}}{%
{\protect \APACyear {2015}}%
}]{%
bertsimas2015data}
\APACinsertmetastar {%
bertsimas2015data}%
\begin{APACrefauthors}%
Bertsimas, D.%
, Gupta, V.%
\BCBL {}\ \BBA {} Paschalidis, I\BPBI C.%
\end{APACrefauthors}%
\unskip\
\newblock
\APACrefYearMonthDay{2015}{}{}.
\newblock
{\BBOQ}\APACrefatitle {Data-driven estimation in equilibrium using inverse
  optimization} {Data-driven estimation in equilibrium using inverse
  optimization}.{\BBCQ}
\newblock
\APACjournalVolNumPages{Mathematical Programming}{153}{2}{595--633}.
\PrintBackRefs{\CurrentBib}

\bibitem [\protect \citeauthoryear {%
Branston%
}{%
Branston%
}{%
{\protect \APACyear {1976}}%
}]{%
branston1976link}
\APACinsertmetastar {%
branston1976link}%
\begin{APACrefauthors}%
Branston, D.%
\end{APACrefauthors}%
\unskip\
\newblock
\APACrefYearMonthDay{1976}{}{}.
\newblock
{\BBOQ}\APACrefatitle {Link capacity functions: A review} {Link capacity
  functions: A review}.{\BBCQ}
\newblock
\APACjournalVolNumPages{Transportation research}{10}{4}{223--236}.
\PrintBackRefs{\CurrentBib}

\bibitem [\protect \citeauthoryear {%
Cantillo%
, Macea%
\BCBL {}\ \BBA {} Jaller%
}{%
Cantillo%
\ \protect \BOthers {.}}{%
{\protect \APACyear {2019}}%
}]{%
cantillo2019assessing}
\APACinsertmetastar {%
cantillo2019assessing}%
\begin{APACrefauthors}%
Cantillo, V.%
, Macea, L\BPBI F.%
\BCBL {}\ \BBA {} Jaller, M.%
\end{APACrefauthors}%
\unskip\
\newblock
\APACrefYearMonthDay{2019}{}{}.
\newblock
{\BBOQ}\APACrefatitle {Assessing vulnerability of transportation networks for
  disaster response operations} {Assessing vulnerability of transportation
  networks for disaster response operations}.{\BBCQ}
\newblock
\APACjournalVolNumPages{Networks and Spatial Economics}{19}{1}{243--273}.
\PrintBackRefs{\CurrentBib}

\bibitem [\protect \citeauthoryear {%
Carpentier%
, Gendreau%
\BCBL {}\ \BBA {} Bastin%
}{%
Carpentier%
\ \protect \BOthers {.}}{%
{\protect \APACyear {2013}}%
}]{%
carpentier2013long}
\APACinsertmetastar {%
carpentier2013long}%
\begin{APACrefauthors}%
Carpentier, P\BHBI L.%
, Gendreau, M.%
\BCBL {}\ \BBA {} Bastin, F.%
\end{APACrefauthors}%
\unskip\
\newblock
\APACrefYearMonthDay{2013}{}{}.
\newblock
{\BBOQ}\APACrefatitle {Long-term management of a hydroelectric multireservoir
  system under uncertainty using the progressive hedging algorithm} {Long-term
  management of a hydroelectric multireservoir system under uncertainty using
  the progressive hedging algorithm}.{\BBCQ}
\newblock
\APACjournalVolNumPages{Water Resources Research}{49}{5}{2812--2827}.
\PrintBackRefs{\CurrentBib}

\bibitem [\protect \citeauthoryear {%
Chan%
, Lee%
\BCBL {}\ \BBA {} Terekhov%
}{%
Chan%
\ \protect \BOthers {.}}{%
{\protect \APACyear {2019}}%
}]{%
chan2019inverse}
\APACinsertmetastar {%
chan2019inverse}%
\begin{APACrefauthors}%
Chan, T\BPBI C.%
, Lee, T.%
\BCBL {}\ \BBA {} Terekhov, D.%
\end{APACrefauthors}%
\unskip\
\newblock
\APACrefYearMonthDay{2019}{}{}.
\newblock
{\BBOQ}\APACrefatitle {Inverse optimization: Closed-form solutions, geometry,
  and goodness of fit} {Inverse optimization: Closed-form solutions, geometry,
  and goodness of fit}.{\BBCQ}
\newblock
\APACjournalVolNumPages{Management Science}{65}{3}{1115--1135}.
\PrintBackRefs{\CurrentBib}

\bibitem [\protect \citeauthoryear {%
Chen%
\ \BBA {} Florian%
}{%
Chen%
\ \BBA {} Florian%
}{%
{\protect \APACyear {1998}}%
}]{%
chen1998congested}
\APACinsertmetastar {%
chen1998congested}%
\begin{APACrefauthors}%
Chen, Y.%
\BCBT {}\ \BBA {} Florian, M.%
\end{APACrefauthors}%
\unskip\
\newblock
\APACrefYearMonthDay{1998}{}{}.
\newblock
{\BBOQ}\APACrefatitle {Congested OD trip demand adjustment problem: bilevel
  programming formulation and optimality conditions} {Congested od trip demand
  adjustment problem: bilevel programming formulation and optimality
  conditions}.{\BBCQ}
\newblock
\BIn{} \APACrefbtitle {Multilevel Optimization: Algorithms and Applications}
  {Multilevel optimization: Algorithms and applications}\ (\BPGS\ 1--22).
\newblock
\APACaddressPublisher{}{Springer}.
\PrintBackRefs{\CurrentBib}

\bibitem [\protect \citeauthoryear {%
Chow%
, Ritchie%
\BCBL {}\ \BBA {} Jeong%
}{%
Chow%
\ \protect \BOthers {.}}{%
{\protect \APACyear {2014}}%
}]{%
chow2014nonlinear}
\APACinsertmetastar {%
chow2014nonlinear}%
\begin{APACrefauthors}%
Chow, J\BPBI Y.%
, Ritchie, S\BPBI G.%
\BCBL {}\ \BBA {} Jeong, K.%
\end{APACrefauthors}%
\unskip\
\newblock
\APACrefYearMonthDay{2014}{}{}.
\newblock
{\BBOQ}\APACrefatitle {Nonlinear inverse optimization for parameter estimation
  of commodity-vehicle-decoupled freight assignment} {Nonlinear inverse
  optimization for parameter estimation of commodity-vehicle-decoupled freight
  assignment}.{\BBCQ}
\newblock
\APACjournalVolNumPages{Transportation Research Part E: Logistics and
  Transportation Review}{67}{}{71--91}.
\PrintBackRefs{\CurrentBib}

\bibitem [\protect \citeauthoryear {%
Chu%
\ \BBA {} Chen%
}{%
Chu%
\ \BBA {} Chen%
}{%
{\protect \APACyear {2016}}%
}]{%
chu2016optimization}
\APACinsertmetastar {%
chu2016optimization}%
\begin{APACrefauthors}%
Chu, J\BPBI C.%
\BCBT {}\ \BBA {} Chen, S\BHBI C.%
\end{APACrefauthors}%
\unskip\
\newblock
\APACrefYearMonthDay{2016}{}{}.
\newblock
{\BBOQ}\APACrefatitle {Optimization of transportation-infrastructure-system
  protection considering weighted connectivity reliability} {Optimization of
  transportation-infrastructure-system protection considering weighted
  connectivity reliability}.{\BBCQ}
\newblock
\APACjournalVolNumPages{Journal of Infrastructure Systems}{22}{1}{04015008}.
\PrintBackRefs{\CurrentBib}

\bibitem [\protect \citeauthoryear {%
Cioca%
, Cioca%
\BCBL {}\ \BBA {} Buraga%
}{%
Cioca%
\ \protect \BOthers {.}}{%
{\protect \APACyear {2007}}%
}]{%
cioca2007spatial}
\APACinsertmetastar {%
cioca2007spatial}%
\begin{APACrefauthors}%
Cioca, M.%
, Cioca, L\BHBI I.%
\BCBL {}\ \BBA {} Buraga, S\BHBI C.%
\end{APACrefauthors}%
\unskip\
\newblock
\APACrefYearMonthDay{2007}{}{}.
\newblock
{\BBOQ}\APACrefatitle {Spatial [Elements] decision support system used in
  disaster management} {Spatial [elements] decision support system used in
  disaster management}.{\BBCQ}
\newblock
\BIn{} \APACrefbtitle {2007 Inaugural IEEE-IES Digital EcoSystems and
  Technologies Conference} {2007 inaugural ieee-ies digital ecosystems and
  technologies conference}\ (\BPGS\ 607--612).
\PrintBackRefs{\CurrentBib}

\bibitem [\protect \citeauthoryear {%
Crainic%
, Fu%
, Gendreau%
, Rei%
\BCBL {}\ \BBA {} Wallace%
}{%
Crainic%
\ \protect \BOthers {.}}{%
{\protect \APACyear {2011}}%
}]{%
crainic2011progressive}
\APACinsertmetastar {%
crainic2011progressive}%
\begin{APACrefauthors}%
Crainic, T\BPBI G.%
, Fu, X.%
, Gendreau, M.%
, Rei, W.%
\BCBL {}\ \BBA {} Wallace, S\BPBI W.%
\end{APACrefauthors}%
\unskip\
\newblock
\APACrefYearMonthDay{2011}{}{}.
\newblock
{\BBOQ}\APACrefatitle {Progressive hedging-based metaheuristics for stochastic
  network design} {Progressive hedging-based metaheuristics for stochastic
  network design}.{\BBCQ}
\newblock
\APACjournalVolNumPages{Networks}{58}{2}{114--124}.
\PrintBackRefs{\CurrentBib}

\bibitem [\protect \citeauthoryear {%
Cuesta%
, Alvear%
, Abreu%
\BCBL {}\ \BBA {} Sili{\'o}%
}{%
Cuesta%
\ \protect \BOthers {.}}{%
{\protect \APACyear {2014}}%
}]{%
cuesta2014real}
\APACinsertmetastar {%
cuesta2014real}%
\begin{APACrefauthors}%
Cuesta, A.%
, Alvear, D.%
, Abreu, O.%
\BCBL {}\ \BBA {} Sili{\'o}, D.%
\end{APACrefauthors}%
\unskip\
\newblock
\APACrefYearMonthDay{2014}{}{}.
\newblock
{\BBOQ}\APACrefatitle {Real-time stochastic evacuation models for decision
  support in actual emergencies} {Real-time stochastic evacuation models for
  decision support in actual emergencies}.{\BBCQ}
\newblock
\APACjournalVolNumPages{Fire Safety Science}{11}{}{1063--1076}.
\PrintBackRefs{\CurrentBib}

\bibitem [\protect \citeauthoryear {%
Dirkse%
\ \BBA {} Ferris%
}{%
Dirkse%
\ \BBA {} Ferris%
}{%
{\protect \APACyear {1995}}%
}]{%
dirkse1995path}
\APACinsertmetastar {%
dirkse1995path}%
\begin{APACrefauthors}%
Dirkse, S\BPBI P.%
\BCBT {}\ \BBA {} Ferris, M\BPBI C.%
\end{APACrefauthors}%
\unskip\
\newblock
\APACrefYearMonthDay{1995}{}{}.
\newblock
{\BBOQ}\APACrefatitle {The path solver: a nommonotone stabilization scheme for
  mixed complementarity problems} {The path solver: a nommonotone stabilization
  scheme for mixed complementarity problems}.{\BBCQ}
\newblock
\APACjournalVolNumPages{Optimization Methods and Software}{5}{2}{123--156}.
\PrintBackRefs{\CurrentBib}

\bibitem [\protect \citeauthoryear {%
D{\"o}yen%
\ \BBA {} Aras%
}{%
D{\"o}yen%
\ \BBA {} Aras%
}{%
{\protect \APACyear {2019}}%
}]{%
doyen2019integrated}
\APACinsertmetastar {%
doyen2019integrated}%
\begin{APACrefauthors}%
D{\"o}yen, A.%
\BCBT {}\ \BBA {} Aras, N.%
\end{APACrefauthors}%
\unskip\
\newblock
\APACrefYearMonthDay{2019}{}{}.
\newblock
{\BBOQ}\APACrefatitle {An Integrated Disaster Preparedness Model for
  Retrofitting and Relief Item Transportation} {An integrated disaster
  preparedness model for retrofitting and relief item transportation}.{\BBCQ}
\newblock
\APACjournalVolNumPages{Networks and Spatial Economics}{19}{4}{1031--1068}.
\PrintBackRefs{\CurrentBib}

\bibitem [\protect \citeauthoryear {%
Eguchi%
\ \protect \BOthers {.}}{%
Eguchi%
\ \protect \BOthers {.}}{%
{\protect \APACyear {1997}}%
}]{%
eguchi1997real}
\APACinsertmetastar {%
eguchi1997real}%
\begin{APACrefauthors}%
Eguchi, R\BPBI T.%
, Goltz, J\BPBI D.%
, Seligson, H\BPBI A.%
, Flores, P\BPBI J.%
, Blais, N\BPBI C.%
, Heaton, T\BPBI H.%
\BCBL {}\ \BBA {} Bortugno, E.%
\end{APACrefauthors}%
\unskip\
\newblock
\APACrefYearMonthDay{1997}{}{}.
\newblock
{\BBOQ}\APACrefatitle {Real-time loss estimation as an emergency response
  decision support system: the early post-earthquake damage assessment tool
  (EPEDAT)} {Real-time loss estimation as an emergency response decision
  support system: the early post-earthquake damage assessment tool
  (epedat)}.{\BBCQ}
\newblock
\APACjournalVolNumPages{Earthquake Spectra}{13}{4}{815--832}.
\PrintBackRefs{\CurrentBib}

\bibitem [\protect \citeauthoryear {%
Facchinei%
\ \BBA {} Pang%
}{%
Facchinei%
\ \BBA {} Pang%
}{%
{\protect \APACyear {2007}}%
}]{%
facchinei2007finite}
\APACinsertmetastar {%
facchinei2007finite}%
\begin{APACrefauthors}%
Facchinei, F.%
\BCBT {}\ \BBA {} Pang, J\BHBI S.%
\end{APACrefauthors}%
\unskip\
\newblock
\APACrefYear{2007}.
\newblock
\APACrefbtitle {Finite-dimensional variational inequalities and complementarity
  problems} {Finite-dimensional variational inequalities and complementarity
  problems}.
\newblock
\APACaddressPublisher{}{Springer Science \& Business Media}.
\PrintBackRefs{\CurrentBib}

\bibitem [\protect \citeauthoryear {%
Fan%
\ \BBA {} Liu%
}{%
Fan%
\ \BBA {} Liu%
}{%
{\protect \APACyear {2010}}%
}]{%
fan2010solving}
\APACinsertmetastar {%
fan2010solving}%
\begin{APACrefauthors}%
Fan, Y.%
\BCBT {}\ \BBA {} Liu, C.%
\end{APACrefauthors}%
\unskip\
\newblock
\APACrefYearMonthDay{2010}{}{}.
\newblock
{\BBOQ}\APACrefatitle {Solving stochastic transportation network protection
  problems using the progressive hedging-based method} {Solving stochastic
  transportation network protection problems using the progressive
  hedging-based method}.{\BBCQ}
\newblock
\APACjournalVolNumPages{Networks and Spatial Economics}{10}{2}{193--208}.
\PrintBackRefs{\CurrentBib}

\bibitem [\protect \citeauthoryear {%
Faturechi%
, Isaac%
, Miller-Hooks%
\BCBL {}\ \BBA {} Feng%
}{%
Faturechi%
\ \protect \BOthers {.}}{%
{\protect \APACyear {2018}}%
}]{%
faturechi2018risk}
\APACinsertmetastar {%
faturechi2018risk}%
\begin{APACrefauthors}%
Faturechi, R.%
, Isaac, S.%
, Miller-Hooks, E.%
\BCBL {}\ \BBA {} Feng, L.%
\end{APACrefauthors}%
\unskip\
\newblock
\APACrefYearMonthDay{2018}{}{}.
\newblock
{\BBOQ}\APACrefatitle {Risk-based models for emergency shelter and exit design
  in buildings} {Risk-based models for emergency shelter and exit design in
  buildings}.{\BBCQ}
\newblock
\APACjournalVolNumPages{Annals of Operations Research}{262}{1}{185--212}.
\PrintBackRefs{\CurrentBib}

\bibitem [\protect \citeauthoryear {%
Faturechi%
\ \BBA {} Miller-Hooks%
}{%
Faturechi%
\ \BBA {} Miller-Hooks%
}{%
{\protect \APACyear {2014}}%
}]{%
faturechi2014travel}
\APACinsertmetastar {%
faturechi2014travel}%
\begin{APACrefauthors}%
Faturechi, R.%
\BCBT {}\ \BBA {} Miller-Hooks, E.%
\end{APACrefauthors}%
\unskip\
\newblock
\APACrefYearMonthDay{2014}{}{}.
\newblock
{\BBOQ}\APACrefatitle {Travel time resilience of roadway networks under
  disaster} {Travel time resilience of roadway networks under disaster}.{\BBCQ}
\newblock
\APACjournalVolNumPages{Transportation research part B:
  methodological}{70}{}{47--64}.
\PrintBackRefs{\CurrentBib}

\bibitem [\protect \citeauthoryear {%
Ferris%
\ \BBA {} Munson%
}{%
Ferris%
\ \BBA {} Munson%
}{%
{\protect \APACyear {2000}}%
}]{%
ferris2000complementarity}
\APACinsertmetastar {%
ferris2000complementarity}%
\begin{APACrefauthors}%
Ferris, M\BPBI C.%
\BCBT {}\ \BBA {} Munson, T\BPBI S.%
\end{APACrefauthors}%
\unskip\
\newblock
\APACrefYearMonthDay{2000}{}{}.
\newblock
{\BBOQ}\APACrefatitle {Complementarity problems in GAMS and the PATH solver}
  {Complementarity problems in gams and the path solver}.{\BBCQ}
\newblock
\APACjournalVolNumPages{Journal of Economic Dynamics and
  Control}{24}{2}{165--188}.
\PrintBackRefs{\CurrentBib}

\bibitem [\protect \citeauthoryear {%
Ferris%
\ \BBA {} Munson%
}{%
Ferris%
\ \BBA {} Munson%
}{%
{\protect \APACyear {2020}}%
}]{%
path_website}
\APACinsertmetastar {%
path_website}%
\begin{APACrefauthors}%
Ferris, M\BPBI C.%
\BCBT {}\ \BBA {} Munson, T\BPBI S.%
\end{APACrefauthors}%
\unskip\
\newblock
\APACrefYearMonthDay{2020}{}{}.
\newblock
\APACrefbtitle {PATH 4.7.} {Path 4.7.}
\newblock
\APAChowpublished {\url{https://www.gams.com/latest/docs/S_PATH.html}}.
\PrintBackRefs{\CurrentBib}

\bibitem [\protect \citeauthoryear {%
Fertier%
, Barthe-Delano{\"e}%
, Montarnal%
, Truptil%
\BCBL {}\ \BBA {} B{\'e}naben%
}{%
Fertier%
\ \protect \BOthers {.}}{%
{\protect \APACyear {2020}}%
}]{%
fertier2020new}
\APACinsertmetastar {%
fertier2020new}%
\begin{APACrefauthors}%
Fertier, A.%
, Barthe-Delano{\"e}, A\BHBI M.%
, Montarnal, A.%
, Truptil, S.%
\BCBL {}\ \BBA {} B{\'e}naben, F.%
\end{APACrefauthors}%
\unskip\
\newblock
\APACrefYearMonthDay{2020}{}{}.
\newblock
{\BBOQ}\APACrefatitle {A new emergency decision support system: the automatic
  interpretation and contextualisation of events to model a crisis situation in
  real-time} {A new emergency decision support system: the automatic
  interpretation and contextualisation of events to model a crisis situation in
  real-time}.{\BBCQ}
\newblock
\APACjournalVolNumPages{Decision Support Systems}{}{}{113260}.
\PrintBackRefs{\CurrentBib}

\bibitem [\protect \citeauthoryear {%
Fikar%
, Gronalt%
\BCBL {}\ \BBA {} Hirsch%
}{%
Fikar%
\ \protect \BOthers {.}}{%
{\protect \APACyear {2016}}%
}]{%
fikar2016decision}
\APACinsertmetastar {%
fikar2016decision}%
\begin{APACrefauthors}%
Fikar, C.%
, Gronalt, M.%
\BCBL {}\ \BBA {} Hirsch, P.%
\end{APACrefauthors}%
\unskip\
\newblock
\APACrefYearMonthDay{2016}{}{}.
\newblock
{\BBOQ}\APACrefatitle {A decision support system for coordinated disaster
  relief distribution} {A decision support system for coordinated disaster
  relief distribution}.{\BBCQ}
\newblock
\APACjournalVolNumPages{Expert Systems with Applications}{57}{}{104--116}.
\PrintBackRefs{\CurrentBib}

\bibitem [\protect \citeauthoryear {%
Fortuny-Amat%
\ \BBA {} McCarl%
}{%
Fortuny-Amat%
\ \BBA {} McCarl%
}{%
{\protect \APACyear {1981}}%
}]{%
fortuny1981representation}
\APACinsertmetastar {%
fortuny1981representation}%
\begin{APACrefauthors}%
Fortuny-Amat, J.%
\BCBT {}\ \BBA {} McCarl, B.%
\end{APACrefauthors}%
\unskip\
\newblock
\APACrefYearMonthDay{1981}{}{}.
\newblock
{\BBOQ}\APACrefatitle {A representation and economic interpretation of a
  two-level programming problem} {A representation and economic interpretation
  of a two-level programming problem}.{\BBCQ}
\newblock
\APACjournalVolNumPages{Journal of the operational Research
  Society}{32}{9}{783--792}.
\PrintBackRefs{\CurrentBib}

\bibitem [\protect \citeauthoryear {%
Gabriel%
, Conejo%
, Fuller%
, Hobbs%
\BCBL {}\ \BBA {} Ruiz%
}{%
Gabriel%
\ \protect \BOthers {.}}{%
{\protect \APACyear {2012}}%
}]{%
gabriel2012complementarity}
\APACinsertmetastar {%
gabriel2012complementarity}%
\begin{APACrefauthors}%
Gabriel, S\BPBI A.%
, Conejo, A\BPBI J.%
, Fuller, J\BPBI D.%
, Hobbs, B\BPBI F.%
\BCBL {}\ \BBA {} Ruiz, C.%
\end{APACrefauthors}%
\unskip\
\newblock
\APACrefYear{2012}.
\newblock
\APACrefbtitle {Complementarity modeling in energy markets} {Complementarity
  modeling in energy markets}\ (\BVOL~180).
\newblock
\APACaddressPublisher{}{Springer Science \& Business Media}.
\PrintBackRefs{\CurrentBib}

\bibitem [\protect \citeauthoryear {%
{GAMS Development Corporation}%
}{%
{GAMS Development Corporation}%
}{%
{\protect \APACyear {2021}}%
}]{%
GAMS_software_34_1}
\APACinsertmetastar {%
GAMS_software_34_1}%
\begin{APACrefauthors}%
{GAMS Development Corporation}.%
\end{APACrefauthors}%
\unskip\
\newblock
\APACrefYearMonthDay{2021}{}{}.
\newblock
\APACrefbtitle {General Algebraic Modeling System (GAMS) Release 34.1.0.}
  {General algebraic modeling system (gams) release 34.1.0.}
\newblock
\APAChowpublished {\url{https://www.gams.com/download/}}.
\PrintBackRefs{\CurrentBib}

\bibitem [\protect \citeauthoryear {%
Gon{\c{c}}alves%
, Finardi%
\BCBL {}\ \BBA {} da Silva%
}{%
Gon{\c{c}}alves%
\ \protect \BOthers {.}}{%
{\protect \APACyear {2012}}%
}]{%
gonccalves2012applying}
\APACinsertmetastar {%
gonccalves2012applying}%
\begin{APACrefauthors}%
Gon{\c{c}}alves, R\BPBI E.%
, Finardi, E\BPBI C.%
\BCBL {}\ \BBA {} da Silva, E\BPBI L.%
\end{APACrefauthors}%
\unskip\
\newblock
\APACrefYearMonthDay{2012}{}{}.
\newblock
{\BBOQ}\APACrefatitle {Applying different decomposition schemes using the
  progressive hedging algorithm to the operation planning problem of a
  hydrothermal system} {Applying different decomposition schemes using the
  progressive hedging algorithm to the operation planning problem of a
  hydrothermal system}.{\BBCQ}
\newblock
\APACjournalVolNumPages{Electric power systems research}{83}{1}{19--27}.
\PrintBackRefs{\CurrentBib}

\bibitem [\protect \citeauthoryear {%
Grass%
\ \BBA {} Fischer%
}{%
Grass%
\ \BBA {} Fischer%
}{%
{\protect \APACyear {2016}}%
}]{%
grass2016two}
\APACinsertmetastar {%
grass2016two}%
\begin{APACrefauthors}%
Grass, E.%
\BCBT {}\ \BBA {} Fischer, K.%
\end{APACrefauthors}%
\unskip\
\newblock
\APACrefYearMonthDay{2016}{}{}.
\newblock
{\BBOQ}\APACrefatitle {Two-stage stochastic programming in disaster management:
  A literature survey} {Two-stage stochastic programming in disaster
  management: A literature survey}.{\BBCQ}
\newblock
\APACjournalVolNumPages{Surveys in Operations Research and Management
  Science}{21}{2}{85--100}.
\PrintBackRefs{\CurrentBib}

\bibitem [\protect \citeauthoryear {%
Gul%
, Denton%
\BCBL {}\ \BBA {} Fowler%
}{%
Gul%
\ \protect \BOthers {.}}{%
{\protect \APACyear {2015}}%
}]{%
gul2015progressive}
\APACinsertmetastar {%
gul2015progressive}%
\begin{APACrefauthors}%
Gul, S.%
, Denton, B\BPBI T.%
\BCBL {}\ \BBA {} Fowler, J\BPBI W.%
\end{APACrefauthors}%
\unskip\
\newblock
\APACrefYearMonthDay{2015}{}{}.
\newblock
{\BBOQ}\APACrefatitle {A progressive hedging approach for surgery planning
  under uncertainty} {A progressive hedging approach for surgery planning under
  uncertainty}.{\BBCQ}
\newblock
\APACjournalVolNumPages{INFORMS Journal on Computing}{27}{4}{755--772}.
\PrintBackRefs{\CurrentBib}

\bibitem [\protect \citeauthoryear {%
Guo%
, Liu%
, Li%
\BCBL {}\ \BBA {} Li%
}{%
Guo%
\ \protect \BOthers {.}}{%
{\protect \APACyear {2017}}%
}]{%
guo2017seismic}
\APACinsertmetastar {%
guo2017seismic}%
\begin{APACrefauthors}%
Guo, A.%
, Liu, Z.%
, Li, S.%
\BCBL {}\ \BBA {} Li, H.%
\end{APACrefauthors}%
\unskip\
\newblock
\APACrefYearMonthDay{2017}{}{}.
\newblock
{\BBOQ}\APACrefatitle {Seismic performance assessment of highway bridge
  networks considering post-disaster traffic demand of a transportation system
  in emergency conditions} {Seismic performance assessment of highway bridge
  networks considering post-disaster traffic demand of a transportation system
  in emergency conditions}.{\BBCQ}
\newblock
\APACjournalVolNumPages{Structure and Infrastructure
  Engineering}{13}{12}{1523--1537}.
\PrintBackRefs{\CurrentBib}

\bibitem [\protect \citeauthoryear {%
Gurobi~Optimization%
}{%
Gurobi~Optimization%
}{%
{\protect \APACyear {2021}}%
}]{%
gurobi_citation}
\APACinsertmetastar {%
gurobi_citation}%
\begin{APACrefauthors}%
Gurobi~Optimization, L.%
\end{APACrefauthors}%
\unskip\
\newblock
\APACrefYearMonthDay{2021}{}{}.
\newblock
\APACrefbtitle {Gurobi Optimizer Reference Manual.} {Gurobi optimizer reference
  manual.}
\newblock
\begin{APACrefURL} \url{http://www.gurobi.com} \end{APACrefURL}
\PrintBackRefs{\CurrentBib}

\bibitem [\protect \citeauthoryear {%
Hart%
\ \protect \BOthers {.}}{%
Hart%
\ \protect \BOthers {.}}{%
{\protect \APACyear {2017}}%
{\protect \APACexlab {{\protect \BCnt {1}}}}}]{%
hart2017pyomo}
\APACinsertmetastar {%
hart2017pyomo}%
\begin{APACrefauthors}%
Hart, W\BPBI E.%
, Laird, C\BPBI D.%
, Watson, J\BHBI P.%
, Woodruff, D\BPBI L.%
, Hackebeil, G\BPBI A.%
, Nicholson, B\BPBI L.%
\BCBL {}\ \BBA {} Siirola, J\BPBI D.%
\end{APACrefauthors}%
\unskip\
\newblock
\APACrefYear{2017{\protect \BCnt {1}}}.
\newblock
\APACrefbtitle {Pyomo--optimization modeling in python} {Pyomo--optimization
  modeling in python}\ (\PrintOrdinal{Second}\ \BEd, \BVOL~67).
\newblock
\APACaddressPublisher{}{Springer Science \& Business Media}.
\PrintBackRefs{\CurrentBib}

\bibitem [\protect \citeauthoryear {%
Hart%
\ \protect \BOthers {.}}{%
Hart%
\ \protect \BOthers {.}}{%
{\protect \APACyear {2017}}%
{\protect \APACexlab {{\protect \BCnt {2}}}}}]{%
hart2017mathematical}
\APACinsertmetastar {%
hart2017mathematical}%
\begin{APACrefauthors}%
Hart, W\BPBI E.%
, Laird, C\BPBI D.%
, Watson, J\BHBI P.%
, Woodruff, D\BPBI L.%
, Hackebeil, G\BPBI A.%
, Nicholson, B\BPBI L.%
\BCBL {}\ \BBA {} Siirola, J\BPBI D.%
\end{APACrefauthors}%
\unskip\
\newblock
\APACrefYear{2017{\protect \BCnt {2}}}.
\newblock
\APACrefbtitle {Pyomo—Optimization Modeling in Python: Second Edition}
  {Pyomo—optimization modeling in python: Second edition}\ (\BVOL~67).
\newblock
\APACaddressPublisher{}{Springer Optimization and Its Applications}.
\PrintBackRefs{\CurrentBib}

\bibitem [\protect \citeauthoryear {%
Hart%
, Watson%
\BCBL {}\ \BBA {} Woodruff%
}{%
Hart%
\ \protect \BOthers {.}}{%
{\protect \APACyear {2011}}%
}]{%
hart2011pyomo}
\APACinsertmetastar {%
hart2011pyomo}%
\begin{APACrefauthors}%
Hart, W\BPBI E.%
, Watson, J\BHBI P.%
\BCBL {}\ \BBA {} Woodruff, D\BPBI L.%
\end{APACrefauthors}%
\unskip\
\newblock
\APACrefYearMonthDay{2011}{}{}.
\newblock
{\BBOQ}\APACrefatitle {Pyomo: modeling and solving mathematical programs in
  Python} {Pyomo: modeling and solving mathematical programs in python}.{\BBCQ}
\newblock
\APACjournalVolNumPages{Mathematical Programming Computation}{3}{3}{219}.
\PrintBackRefs{\CurrentBib}

\bibitem [\protect \citeauthoryear {%
Horita%
\ \BBA {} de Albuquerque%
}{%
Horita%
\ \BBA {} de Albuquerque%
}{%
{\protect \APACyear {2013}}%
}]{%
horita2013approach}
\APACinsertmetastar {%
horita2013approach}%
\begin{APACrefauthors}%
Horita, F\BPBI E.%
\BCBT {}\ \BBA {} de Albuquerque, J\BPBI P.%
\end{APACrefauthors}%
\unskip\
\newblock
\APACrefYearMonthDay{2013}{}{}.
\newblock
{\BBOQ}\APACrefatitle {An approach to support decision-making in disaster
  management based on volunteer geographic information (VGI) and spatial
  decision support systems (SDSS).} {An approach to support decision-making in
  disaster management based on volunteer geographic information (vgi) and
  spatial decision support systems (sdss).}{\BBCQ}
\newblock
\BIn{} \APACrefbtitle {ISCRAM.} {Iscram.}
\PrintBackRefs{\CurrentBib}

\bibitem [\protect \citeauthoryear {%
Horita%
, de Albuquerque%
, Degrossi%
, Mendiondo%
\BCBL {}\ \BBA {} Ueyama%
}{%
Horita%
\ \protect \BOthers {.}}{%
{\protect \APACyear {2015}}%
}]{%
horita2015development}
\APACinsertmetastar {%
horita2015development}%
\begin{APACrefauthors}%
Horita, F\BPBI E.%
, de Albuquerque, J\BPBI P.%
, Degrossi, L\BPBI C.%
, Mendiondo, E\BPBI M.%
\BCBL {}\ \BBA {} Ueyama, J.%
\end{APACrefauthors}%
\unskip\
\newblock
\APACrefYearMonthDay{2015}{}{}.
\newblock
{\BBOQ}\APACrefatitle {Development of a spatial decision support system for
  flood risk management in Brazil that combines volunteered geographic
  information with wireless sensor networks} {Development of a spatial decision
  support system for flood risk management in brazil that combines volunteered
  geographic information with wireless sensor networks}.{\BBCQ}
\newblock
\APACjournalVolNumPages{Computers \& Geosciences}{80}{}{84--94}.
\PrintBackRefs{\CurrentBib}

\bibitem [\protect \citeauthoryear {%
Hunter%
}{%
Hunter%
}{%
{\protect \APACyear {2007}}%
}]{%
Hunter:2007}
\APACinsertmetastar {%
Hunter:2007}%
\begin{APACrefauthors}%
Hunter, J\BPBI D.%
\end{APACrefauthors}%
\unskip\
\newblock
\APACrefYearMonthDay{2007}{}{}.
\newblock
{\BBOQ}\APACrefatitle {Matplotlib: A 2D graphics environment} {Matplotlib: A 2d
  graphics environment}.{\BBCQ}
\newblock
\APACjournalVolNumPages{Computing In Science \& Engineering}{9}{3}{90--95}.
\newblock
\begin{APACrefDOI} \doi{10.1109/MCSE.2007.55} \end{APACrefDOI}
\PrintBackRefs{\CurrentBib}

\bibitem [\protect \citeauthoryear {%
Hvattum%
\ \BBA {} L{\o}kketangen%
}{%
Hvattum%
\ \BBA {} L{\o}kketangen%
}{%
{\protect \APACyear {2009}}%
}]{%
hvattum2009using}
\APACinsertmetastar {%
hvattum2009using}%
\begin{APACrefauthors}%
Hvattum, L\BPBI M.%
\BCBT {}\ \BBA {} L{\o}kketangen, A.%
\end{APACrefauthors}%
\unskip\
\newblock
\APACrefYearMonthDay{2009}{}{}.
\newblock
{\BBOQ}\APACrefatitle {Using scenario trees and progressive hedging for
  stochastic inventory routing problems} {Using scenario trees and progressive
  hedging for stochastic inventory routing problems}.{\BBCQ}
\newblock
\APACjournalVolNumPages{Journal of Heuristics}{15}{6}{527}.
\PrintBackRefs{\CurrentBib}

\bibitem [\protect \citeauthoryear {%
Kaviani%
, Thompson%
, Rajabifard%
, Griffin%
\BCBL {}\ \BBA {} Chen%
}{%
Kaviani%
\ \protect \BOthers {.}}{%
{\protect \APACyear {2015}}%
}]{%
kaviani2015decision}
\APACinsertmetastar {%
kaviani2015decision}%
\begin{APACrefauthors}%
Kaviani, A.%
, Thompson, R\BPBI G.%
, Rajabifard, A.%
, Griffin, G.%
\BCBL {}\ \BBA {} Chen, Y.%
\end{APACrefauthors}%
\unskip\
\newblock
\APACrefYearMonthDay{2015}{}{}.
\newblock
{\BBOQ}\APACrefatitle {A decision support system for improving the management
  of traffic networks during disasters} {A decision support system for
  improving the management of traffic networks during disasters}.{\BBCQ}
\newblock
\BIn{} \APACrefbtitle {37th Australasian Transport Research Forum (ATRF),
  Sydney, New South Wales, Australia.} {37th australasian transport research
  forum (atrf), sydney, new south wales, australia.}
\PrintBackRefs{\CurrentBib}

\bibitem [\protect \citeauthoryear {%
Kureshi%
, Theodoropoulos%
, Mangina%
, O'Hare%
\BCBL {}\ \BBA {} Roche%
}{%
Kureshi%
\ \protect \BOthers {.}}{%
{\protect \APACyear {2015}}%
}]{%
kureshi2015towards}
\APACinsertmetastar {%
kureshi2015towards}%
\begin{APACrefauthors}%
Kureshi, I.%
, Theodoropoulos, G.%
, Mangina, E.%
, O'Hare, G.%
\BCBL {}\ \BBA {} Roche, J.%
\end{APACrefauthors}%
\unskip\
\newblock
\APACrefYearMonthDay{2015}{}{}.
\newblock
{\BBOQ}\APACrefatitle {Towards an info-symbiotic decision support system for
  disaster risk management} {Towards an info-symbiotic decision support system
  for disaster risk management}.{\BBCQ}
\newblock
\BIn{} \APACrefbtitle {2015 IEEE/ACM 19th International Symposium on
  Distributed Simulation and Real Time Applications (DS-RT)} {2015 ieee/acm
  19th international symposium on distributed simulation and real time
  applications (ds-rt)}\ (\BPGS\ 85--91).
\PrintBackRefs{\CurrentBib}

\bibitem [\protect \citeauthoryear {%
Lamghari%
\ \BBA {} Dimitrakopoulos%
}{%
Lamghari%
\ \BBA {} Dimitrakopoulos%
}{%
{\protect \APACyear {2016}}%
}]{%
lamghari2016progressive}
\APACinsertmetastar {%
lamghari2016progressive}%
\begin{APACrefauthors}%
Lamghari, A.%
\BCBT {}\ \BBA {} Dimitrakopoulos, R.%
\end{APACrefauthors}%
\unskip\
\newblock
\APACrefYearMonthDay{2016}{}{}.
\newblock
{\BBOQ}\APACrefatitle {Progressive hedging applied as a metaheuristic to
  schedule production in open-pit mines accounting for reserve uncertainty}
  {Progressive hedging applied as a metaheuristic to schedule production in
  open-pit mines accounting for reserve uncertainty}.{\BBCQ}
\newblock
\APACjournalVolNumPages{European Journal of Operational
  Research}{253}{3}{843--855}.
\PrintBackRefs{\CurrentBib}

\bibitem [\protect \citeauthoryear {%
J.~Lu%
, Atamturktur%
\BCBL {}\ \BBA {} Huang%
}{%
J.~Lu%
\ \protect \BOthers {.}}{%
{\protect \APACyear {2016}}%
}]{%
lu2016bi}
\APACinsertmetastar {%
lu2016bi}%
\begin{APACrefauthors}%
Lu, J.%
, Atamturktur, S.%
\BCBL {}\ \BBA {} Huang, Y.%
\end{APACrefauthors}%
\unskip\
\newblock
\APACrefYearMonthDay{2016}{}{}.
\newblock
{\BBOQ}\APACrefatitle {Bi-level resource allocation framework for retrofitting
  bridges in a transportation network} {Bi-level resource allocation framework
  for retrofitting bridges in a transportation network}.{\BBCQ}
\newblock
\APACjournalVolNumPages{Transportation Research Record}{2550}{1}{31--37}.
\PrintBackRefs{\CurrentBib}

\bibitem [\protect \citeauthoryear {%
J.~Lu%
, Gupte%
\BCBL {}\ \BBA {} Huang%
}{%
J.~Lu%
\ \protect \BOthers {.}}{%
{\protect \APACyear {2018}}%
}]{%
lu2018mean}
\APACinsertmetastar {%
lu2018mean}%
\begin{APACrefauthors}%
Lu, J.%
, Gupte, A.%
\BCBL {}\ \BBA {} Huang, Y.%
\end{APACrefauthors}%
\unskip\
\newblock
\APACrefYearMonthDay{2018}{}{}.
\newblock
{\BBOQ}\APACrefatitle {A mean-risk mixed integer nonlinear program for
  transportation network protection} {A mean-risk mixed integer nonlinear
  program for transportation network protection}.{\BBCQ}
\newblock
\APACjournalVolNumPages{European Journal of Operational
  Research}{265}{1}{277--289}.
\PrintBackRefs{\CurrentBib}

\bibitem [\protect \citeauthoryear {%
Z.~Lu%
, Meng%
\BCBL {}\ \BBA {} Gomes%
}{%
Z.~Lu%
\ \protect \BOthers {.}}{%
{\protect \APACyear {2016}}%
}]{%
lu2016estimating}
\APACinsertmetastar {%
lu2016estimating}%
\begin{APACrefauthors}%
Lu, Z.%
, Meng, Q.%
\BCBL {}\ \BBA {} Gomes, G.%
\end{APACrefauthors}%
\unskip\
\newblock
\APACrefYearMonthDay{2016}{}{}.
\newblock
{\BBOQ}\APACrefatitle {Estimating link travel time functions for heterogeneous
  traffic flows on freeways} {Estimating link travel time functions for
  heterogeneous traffic flows on freeways}.{\BBCQ}
\newblock
\APACjournalVolNumPages{Journal of Advanced Transportation}{50}{8}{1683--1698}.
\PrintBackRefs{\CurrentBib}

\bibitem [\protect \citeauthoryear {%
Luathep%
, Sumalee%
, Lam%
, Li%
\BCBL {}\ \BBA {} Lo%
}{%
Luathep%
\ \protect \BOthers {.}}{%
{\protect \APACyear {2011}}%
}]{%
luathep2011global}
\APACinsertmetastar {%
luathep2011global}%
\begin{APACrefauthors}%
Luathep, P.%
, Sumalee, A.%
, Lam, W\BPBI H.%
, Li, Z\BHBI C.%
\BCBL {}\ \BBA {} Lo, H\BPBI K.%
\end{APACrefauthors}%
\unskip\
\newblock
\APACrefYearMonthDay{2011}{}{}.
\newblock
{\BBOQ}\APACrefatitle {Global optimization method for mixed transportation
  network design problem: a mixed-integer linear programming approach} {Global
  optimization method for mixed transportation network design problem: a
  mixed-integer linear programming approach}.{\BBCQ}
\newblock
\APACjournalVolNumPages{Transportation Research Part B:
  Methodological}{45}{5}{808--827}.
\PrintBackRefs{\CurrentBib}

\bibitem [\protect \citeauthoryear {%
Marcotte%
\ \BBA {} Patriksson%
}{%
Marcotte%
\ \BBA {} Patriksson%
}{%
{\protect \APACyear {2007}}%
}]{%
marcotte2007traffic}
\APACinsertmetastar {%
marcotte2007traffic}%
\begin{APACrefauthors}%
Marcotte, P.%
\BCBT {}\ \BBA {} Patriksson, M.%
\end{APACrefauthors}%
\unskip\
\newblock
\APACrefYearMonthDay{2007}{}{}.
\newblock
{\BBOQ}\APACrefatitle {Traffic equilibrium} {Traffic equilibrium}.{\BBCQ}
\newblock
\APACjournalVolNumPages{Handbooks in Operations Research and Management
  Science}{14}{}{623--713}.
\PrintBackRefs{\CurrentBib}

\bibitem [\protect \citeauthoryear {%
MATLAB%
}{%
MATLAB%
}{%
{\protect \APACyear {2020}}%
}]{%
MATLAB:2020a}
\APACinsertmetastar {%
MATLAB:2020a}%
\begin{APACrefauthors}%
MATLAB.%
\end{APACrefauthors}%
\unskip\
\newblock
\APACrefYear{2020}.
\newblock
\APACrefbtitle {Version 9.8.0.1417392 (R2020a) Update 4} {Version 9.8.0.1417392
  (r2020a) update 4}.
\newblock
\APACaddressPublisher{Natick, Massachusetts}{The MathWorks, Inc.}
\PrintBackRefs{\CurrentBib}

\bibitem [\protect \citeauthoryear {%
McKinney%
\ \protect \BOthers {.}}{%
McKinney%
\ \protect \BOthers {.}}{%
{\protect \APACyear {2010}}%
}]{%
mckinney2010data}
\APACinsertmetastar {%
mckinney2010data}%
\begin{APACrefauthors}%
McKinney, W.%
\BCBT {}\ \BOthersPeriod {.}
\end{APACrefauthors}%
\unskip\
\newblock
\APACrefYearMonthDay{2010}{}{}.
\newblock
{\BBOQ}\APACrefatitle {Data structures for statistical computing in python}
  {Data structures for statistical computing in python}.{\BBCQ}
\newblock
\BIn{} \APACrefbtitle {Proceedings of the 9th Python in Science Conference}
  {Proceedings of the 9th python in science conference}\ (\BVOL~445, \BPGS\
  51--56).
\PrintBackRefs{\CurrentBib}

\bibitem [\protect \citeauthoryear {%
Mohammadi%
, Ghomi%
\BCBL {}\ \BBA {} Jolai%
}{%
Mohammadi%
\ \protect \BOthers {.}}{%
{\protect \APACyear {2016}}%
}]{%
mohammadi2016prepositioning}
\APACinsertmetastar {%
mohammadi2016prepositioning}%
\begin{APACrefauthors}%
Mohammadi, R.%
, Ghomi, S\BPBI F.%
\BCBL {}\ \BBA {} Jolai, F.%
\end{APACrefauthors}%
\unskip\
\newblock
\APACrefYearMonthDay{2016}{}{}.
\newblock
{\BBOQ}\APACrefatitle {Prepositioning emergency earthquake response supplies: A
  new multi-objective particle swarm optimization algorithm} {Prepositioning
  emergency earthquake response supplies: A new multi-objective particle swarm
  optimization algorithm}.{\BBCQ}
\newblock
\APACjournalVolNumPages{Applied Mathematical Modelling}{40}{9-10}{5183--5199}.
\PrintBackRefs{\CurrentBib}

\bibitem [\protect \citeauthoryear {%
Mulvey%
\ \BBA {} Vladimirou%
}{%
Mulvey%
\ \BBA {} Vladimirou%
}{%
{\protect \APACyear {1991}}%
}]{%
mulvey1991applying}
\APACinsertmetastar {%
mulvey1991applying}%
\begin{APACrefauthors}%
Mulvey, J\BPBI M.%
\BCBT {}\ \BBA {} Vladimirou, H.%
\end{APACrefauthors}%
\unskip\
\newblock
\APACrefYearMonthDay{1991}{}{}.
\newblock
{\BBOQ}\APACrefatitle {Applying the progressive hedging algorithm to stochastic
  generalized networks} {Applying the progressive hedging algorithm to
  stochastic generalized networks}.{\BBCQ}
\newblock
\APACjournalVolNumPages{Annals of Operations Research}{31}{1}{399--424}.
\PrintBackRefs{\CurrentBib}

\bibitem [\protect \citeauthoryear {%
Murty%
}{%
Murty%
}{%
{\protect \APACyear {1983}}%
}]{%
murty1983linear}
\APACinsertmetastar {%
murty1983linear}%
\begin{APACrefauthors}%
Murty, K\BPBI G.%
\end{APACrefauthors}%
\unskip\
\newblock
\APACrefYear{1983}.
\newblock
\APACrefbtitle {Linear programming} {Linear programming}.
\newblock
\APACaddressPublisher{}{Springer}.
\PrintBackRefs{\CurrentBib}

\bibitem [\protect \citeauthoryear {%
NCEI%
}{%
NCEI%
}{%
{\protect \APACyear {2021}}%
}]{%
noaa_centers}
\APACinsertmetastar {%
noaa_centers}%
\begin{APACrefauthors}%
NCEI.%
\end{APACrefauthors}%
\unskip\
\newblock
\APACrefYearMonthDay{2021}{}{}.
\newblock
\APACrefbtitle {U.S. Billion-Dollar Weather and Climate Disasters.} {U.s.
  billion-dollar weather and climate disasters.}
\newblock
\APAChowpublished {\url{https://www.ncdc.noaa.gov/billions/}, DOI:
  \url{10.25921/stkw-7w73}}.
\PrintBackRefs{\CurrentBib}

\bibitem [\protect \citeauthoryear {%
Nguyen%
\ \BBA {} Dupuis%
}{%
Nguyen%
\ \BBA {} Dupuis%
}{%
{\protect \APACyear {1984}}%
}]{%
nguyen1984efficient}
\APACinsertmetastar {%
nguyen1984efficient}%
\begin{APACrefauthors}%
Nguyen, S.%
\BCBT {}\ \BBA {} Dupuis, C.%
\end{APACrefauthors}%
\unskip\
\newblock
\APACrefYearMonthDay{1984}{}{}.
\newblock
{\BBOQ}\APACrefatitle {An efficient method for computing traffic equilibria in
  networks with asymmetric transportation costs} {An efficient method for
  computing traffic equilibria in networks with asymmetric transportation
  costs}.{\BBCQ}
\newblock
\APACjournalVolNumPages{Transportation Science}{18}{2}{185--202}.
\PrintBackRefs{\CurrentBib}

\bibitem [\protect \citeauthoryear {%
Noyan%
}{%
Noyan%
}{%
{\protect \APACyear {2012}}%
}]{%
noyan2012risk}
\APACinsertmetastar {%
noyan2012risk}%
\begin{APACrefauthors}%
Noyan, N.%
\end{APACrefauthors}%
\unskip\
\newblock
\APACrefYearMonthDay{2012}{}{}.
\newblock
{\BBOQ}\APACrefatitle {Risk-averse two-stage stochastic programming with an
  application to disaster management} {Risk-averse two-stage stochastic
  programming with an application to disaster management}.{\BBCQ}
\newblock
\APACjournalVolNumPages{Computers \& Operations Research}{39}{3}{541--559}.
\PrintBackRefs{\CurrentBib}

\bibitem [\protect \citeauthoryear {%
Oliphant%
}{%
Oliphant%
}{%
{\protect \APACyear {2006}}%
}]{%
numpy_citation}
\APACinsertmetastar {%
numpy_citation}%
\begin{APACrefauthors}%
Oliphant, T\BPBI E.%
\end{APACrefauthors}%
\unskip\
\newblock
\APACrefYear{2006}.
\newblock
\APACrefbtitle {A guide to NumPy} {A guide to numpy}.
\newblock
\APACaddressPublisher{}{USA: Trelgol Publishing}.
\PrintBackRefs{\CurrentBib}

\bibitem [\protect \citeauthoryear {%
Otsuka%
, Work%
\BCBL {}\ \BBA {} Song%
}{%
Otsuka%
\ \protect \BOthers {.}}{%
{\protect \APACyear {2016}}%
}]{%
otsuka2016estimating}
\APACinsertmetastar {%
otsuka2016estimating}%
\begin{APACrefauthors}%
Otsuka, R\BPBI P.%
, Work, D\BPBI B.%
\BCBL {}\ \BBA {} Song, J.%
\end{APACrefauthors}%
\unskip\
\newblock
\APACrefYearMonthDay{2016}{}{}.
\newblock
{\BBOQ}\APACrefatitle {Estimating post-disaster traffic conditions using
  real-time data streams} {Estimating post-disaster traffic conditions using
  real-time data streams}.{\BBCQ}
\newblock
\APACjournalVolNumPages{Structure and Infrastructure
  Engineering}{12}{8}{904--917}.
\PrintBackRefs{\CurrentBib}

\bibitem [\protect \citeauthoryear {%
Palsson%
\ \BBA {} Ravn%
}{%
Palsson%
\ \BBA {} Ravn%
}{%
{\protect \APACyear {1994}}%
}]{%
palsson1994stochastic}
\APACinsertmetastar {%
palsson1994stochastic}%
\begin{APACrefauthors}%
Palsson, O\BPBI P.%
\BCBT {}\ \BBA {} Ravn, H\BPBI F.%
\end{APACrefauthors}%
\unskip\
\newblock
\APACrefYearMonthDay{1994}{}{}.
\newblock
{\BBOQ}\APACrefatitle {Stochastic heat storage problem—solved by the
  progressive hedging algorithm} {Stochastic heat storage problem—solved by
  the progressive hedging algorithm}.{\BBCQ}
\newblock
\APACjournalVolNumPages{Energy conversion and management}{35}{12}{1157--1171}.
\PrintBackRefs{\CurrentBib}

\bibitem [\protect \citeauthoryear {%
Programme%
}{%
Programme%
}{%
{\protect \APACyear {2019}}%
}]{%
WFP_link}
\APACinsertmetastar {%
WFP_link}%
\begin{APACrefauthors}%
Programme, W\BPBI F.%
\end{APACrefauthors}%
\unskip\
\newblock
\APACrefYearMonthDay{2019}{}{}.
\newblock
\APACrefbtitle {2.3 Nepal Road Network.} {2.3 nepal road network.}
\newblock
\APAChowpublished
  {\url{https://dlca.logcluster.org/display/public/DLCA/2.3+Nepal+Road+Network}}.
\PrintBackRefs{\CurrentBib}

\bibitem [\protect \citeauthoryear {%
Ratliff%
, Jin%
, Konstantakopoulos%
, Spanos%
\BCBL {}\ \BBA {} Sastry%
}{%
Ratliff%
\ \protect \BOthers {.}}{%
{\protect \APACyear {2014}}%
}]{%
ratliff2014social}
\APACinsertmetastar {%
ratliff2014social}%
\begin{APACrefauthors}%
Ratliff, L\BPBI J.%
, Jin, M.%
, Konstantakopoulos, I\BPBI C.%
, Spanos, C.%
\BCBL {}\ \BBA {} Sastry, S\BPBI S.%
\end{APACrefauthors}%
\unskip\
\newblock
\APACrefYearMonthDay{2014}{}{}.
\newblock
{\BBOQ}\APACrefatitle {Social game for building energy efficiency: Incentive
  design} {Social game for building energy efficiency: Incentive
  design}.{\BBCQ}
\newblock
\BIn{} \APACrefbtitle {2014 52nd Annual Allerton Conference on Communication,
  Control, and Computing (Allerton)} {2014 52nd annual allerton conference on
  communication, control, and computing (allerton)}\ (\BPGS\ 1011--1018).
\PrintBackRefs{\CurrentBib}

\bibitem [\protect \citeauthoryear {%
Rockafellar%
\ \BBA {} Wets%
}{%
Rockafellar%
\ \BBA {} Wets%
}{%
{\protect \APACyear {1991}}%
}]{%
rockafellar1991scenarios}
\APACinsertmetastar {%
rockafellar1991scenarios}%
\begin{APACrefauthors}%
Rockafellar, R\BPBI T.%
\BCBT {}\ \BBA {} Wets, R\BPBI J\BHBI B.%
\end{APACrefauthors}%
\unskip\
\newblock
\APACrefYearMonthDay{1991}{}{}.
\newblock
{\BBOQ}\APACrefatitle {Scenarios and policy aggregation in optimization under
  uncertainty} {Scenarios and policy aggregation in optimization under
  uncertainty}.{\BBCQ}
\newblock
\APACjournalVolNumPages{Mathematics of operations research}{16}{1}{119--147}.
\PrintBackRefs{\CurrentBib}

\bibitem [\protect \citeauthoryear {%
Rodr{\'\i}guez-Esp{\'\i}ndola%
, Albores%
\BCBL {}\ \BBA {} Brewster%
}{%
Rodr{\'\i}guez-Esp{\'\i}ndola%
\ \protect \BOthers {.}}{%
{\protect \APACyear {2018}}%
}]{%
rodriguez2018disaster}
\APACinsertmetastar {%
rodriguez2018disaster}%
\begin{APACrefauthors}%
Rodr{\'\i}guez-Esp{\'\i}ndola, O.%
, Albores, P.%
\BCBL {}\ \BBA {} Brewster, C.%
\end{APACrefauthors}%
\unskip\
\newblock
\APACrefYearMonthDay{2018}{}{}.
\newblock
{\BBOQ}\APACrefatitle {Disaster preparedness in humanitarian logistics: A
  collaborative approach for resource management in floods} {Disaster
  preparedness in humanitarian logistics: A collaborative approach for resource
  management in floods}.{\BBCQ}
\newblock
\APACjournalVolNumPages{European Journal of Operational
  Research}{264}{3}{978--993}.
\PrintBackRefs{\CurrentBib}

\bibitem [\protect \citeauthoryear {%
Ryan%
, Wets%
, Woodruff%
, Silva-Monroy%
\BCBL {}\ \BBA {} Watson%
}{%
Ryan%
\ \protect \BOthers {.}}{%
{\protect \APACyear {2013}}%
}]{%
ryan2013toward}
\APACinsertmetastar {%
ryan2013toward}%
\begin{APACrefauthors}%
Ryan, S\BPBI M.%
, Wets, R\BPBI J\BHBI B.%
, Woodruff, D\BPBI L.%
, Silva-Monroy, C.%
\BCBL {}\ \BBA {} Watson, J\BHBI P.%
\end{APACrefauthors}%
\unskip\
\newblock
\APACrefYearMonthDay{2013}{}{}.
\newblock
{\BBOQ}\APACrefatitle {Toward scalable, parallel progressive hedging for
  stochastic unit commitment} {Toward scalable, parallel progressive hedging
  for stochastic unit commitment}.{\BBCQ}
\newblock
\BIn{} \APACrefbtitle {Power and Energy Society General Meeting (PES), 2013
  IEEE} {Power and energy society general meeting (pes), 2013 ieee}\ (\BPGS\
  1--5).
\PrintBackRefs{\CurrentBib}

\bibitem [\protect \citeauthoryear {%
Sahebjamnia%
, Torabi%
\BCBL {}\ \BBA {} Mansouri%
}{%
Sahebjamnia%
\ \protect \BOthers {.}}{%
{\protect \APACyear {2017}}%
}]{%
sahebjamnia2017hybrid}
\APACinsertmetastar {%
sahebjamnia2017hybrid}%
\begin{APACrefauthors}%
Sahebjamnia, N.%
, Torabi, S\BPBI A.%
\BCBL {}\ \BBA {} Mansouri, S\BPBI A.%
\end{APACrefauthors}%
\unskip\
\newblock
\APACrefYearMonthDay{2017}{}{}.
\newblock
{\BBOQ}\APACrefatitle {A hybrid decision support system for managing
  humanitarian relief chains} {A hybrid decision support system for managing
  humanitarian relief chains}.{\BBCQ}
\newblock
\APACjournalVolNumPages{Decision Support Systems}{95}{}{12--26}.
\PrintBackRefs{\CurrentBib}

\bibitem [\protect \citeauthoryear {%
Schult%
\ \BBA {} Swart%
}{%
Schult%
\ \BBA {} Swart%
}{%
{\protect \APACyear {2008}}%
}]{%
schult2008exploring}
\APACinsertmetastar {%
schult2008exploring}%
\begin{APACrefauthors}%
Schult, D\BPBI A.%
\BCBT {}\ \BBA {} Swart, P.%
\end{APACrefauthors}%
\unskip\
\newblock
\APACrefYearMonthDay{2008}{}{}.
\newblock
{\BBOQ}\APACrefatitle {Exploring network structure, dynamics, and function
  using NetworkX} {Exploring network structure, dynamics, and function using
  networkx}.{\BBCQ}
\newblock
\BIn{} \APACrefbtitle {Proceedings of the 7th Python in science conferences
  (SciPy 2008)} {Proceedings of the 7th python in science conferences (scipy
  2008)}\ (\BVOL\ 2008, \BPGS\ 11--16).
\PrintBackRefs{\CurrentBib}

\bibitem [\protect \citeauthoryear {%
Sheffi%
}{%
Sheffi%
}{%
{\protect \APACyear {1985}}%
}]{%
sheffi_urban_transport}
\APACinsertmetastar {%
sheffi_urban_transport}%
\begin{APACrefauthors}%
Sheffi, Y.%
\end{APACrefauthors}%
\unskip\
\newblock
\APACrefYear{1985}.
\newblock
\APACrefbtitle {Urban Transportation Networks: Equilibrium Analysis with
  Mathematical Programming Methods} {Urban transportation networks: Equilibrium
  analysis with mathematical programming methods}.
\newblock
\APACaddressPublisher{}{Prentice-Hall, Inc.}
\PrintBackRefs{\CurrentBib}

\bibitem [\protect \citeauthoryear {%
Siri%
, Siri%
\BCBL {}\ \BBA {} Sacone%
}{%
Siri%
\ \protect \BOthers {.}}{%
{\protect \APACyear {2020}}%
}]{%
siri2020progressive}
\APACinsertmetastar {%
siri2020progressive}%
\begin{APACrefauthors}%
Siri, E.%
, Siri, S.%
\BCBL {}\ \BBA {} Sacone, S.%
\end{APACrefauthors}%
\unskip\
\newblock
\APACrefYearMonthDay{2020}{}{}.
\newblock
{\BBOQ}\APACrefatitle {A progressive traffic assignment procedure on networks
  affected by disruptive events} {A progressive traffic assignment procedure on
  networks affected by disruptive events}.{\BBCQ}
\newblock
\BIn{} \APACrefbtitle {2020 European Control Conference (ECC)} {2020 european
  control conference (ecc)}\ (\BPGS\ 130--135).
\PrintBackRefs{\CurrentBib}

\bibitem [\protect \citeauthoryear {%
Tai%
, Kizhakkedath%
, Lin%
, Tiong%
\BCBL {}\ \BBA {} Sim%
}{%
Tai%
\ \protect \BOthers {.}}{%
{\protect \APACyear {2013}}%
}]{%
tai2013identifying}
\APACinsertmetastar {%
tai2013identifying}%
\begin{APACrefauthors}%
Tai, K.%
, Kizhakkedath, A.%
, Lin, J.%
, Tiong, R.%
\BCBL {}\ \BBA {} Sim, M.%
\end{APACrefauthors}%
\unskip\
\newblock
\APACrefYearMonthDay{2013}{}{}.
\newblock
{\BBOQ}\APACrefatitle {Identifying extreme risks in critical infrastructure
  interdependencies} {Identifying extreme risks in critical infrastructure
  interdependencies}.{\BBCQ}
\newblock
\BIn{} \APACrefbtitle {Proceedings of the International Symposium of Next
  Generation Infrastructure} {Proceedings of the international symposium of
  next generation infrastructure}\ (\BPGS\ 1--4).
\PrintBackRefs{\CurrentBib}

\bibitem [\protect \citeauthoryear {%
Thai%
\ \BBA {} Bayen%
}{%
Thai%
\ \BBA {} Bayen%
}{%
{\protect \APACyear {2018}}%
}]{%
thai2018imputing}
\APACinsertmetastar {%
thai2018imputing}%
\begin{APACrefauthors}%
Thai, J.%
\BCBT {}\ \BBA {} Bayen, A\BPBI M.%
\end{APACrefauthors}%
\unskip\
\newblock
\APACrefYearMonthDay{2018}{}{}.
\newblock
{\BBOQ}\APACrefatitle {Imputing a variational inequality function or a convex
  objective function: A robust approach} {Imputing a variational inequality
  function or a convex objective function: A robust approach}.{\BBCQ}
\newblock
\APACjournalVolNumPages{Journal of Mathematical Analysis and
  Applications}{457}{2}{1675--1695}.
\PrintBackRefs{\CurrentBib}

\bibitem [\protect \citeauthoryear {%
Thai%
, Hariss%
\BCBL {}\ \BBA {} Bayen%
}{%
Thai%
\ \protect \BOthers {.}}{%
{\protect \APACyear {2015}}%
}]{%
thai2015multi}
\APACinsertmetastar {%
thai2015multi}%
\begin{APACrefauthors}%
Thai, J.%
, Hariss, R.%
\BCBL {}\ \BBA {} Bayen, A.%
\end{APACrefauthors}%
\unskip\
\newblock
\APACrefYearMonthDay{2015}{}{}.
\newblock
{\BBOQ}\APACrefatitle {A multi-convex approach to latency inference and control
  in traffic equilibria from sparse data} {A multi-convex approach to latency
  inference and control in traffic equilibria from sparse data}.{\BBCQ}
\newblock
\BIn{} \APACrefbtitle {2015 American Control Conference (ACC)} {2015 american
  control conference (acc)}\ (\BPGS\ 689--695).
\PrintBackRefs{\CurrentBib}

\bibitem [\protect \citeauthoryear {%
Todini%
}{%
Todini%
}{%
{\protect \APACyear {1999}}%
}]{%
todini1999operational}
\APACinsertmetastar {%
todini1999operational}%
\begin{APACrefauthors}%
Todini, E.%
\end{APACrefauthors}%
\unskip\
\newblock
\APACrefYearMonthDay{1999}{}{}.
\newblock
{\BBOQ}\APACrefatitle {An operational decision support system for flood risk
  mapping, forecasting and management} {An operational decision support system
  for flood risk mapping, forecasting and management}.{\BBCQ}
\newblock
\APACjournalVolNumPages{Urban Water}{1}{2}{131--143}.
\PrintBackRefs{\CurrentBib}

\bibitem [\protect \citeauthoryear {%
US~Department~of Commerce%
}{%
US~Department~of Commerce%
}{%
{\protect \APACyear {1964}}%
}]{%
BPR_func}
\APACinsertmetastar {%
BPR_func}%
\begin{APACrefauthors}%
US~Department~of Commerce, U\BPBI P\BPBI D.%
\end{APACrefauthors}%
\unskip\
\newblock
\APACrefYearMonthDay{1964}{}{}.
\newblock
\APACrefbtitle {Bureau of Public Roads: Traffic assignment manual.} {Bureau of
  public roads: Traffic assignment manual.}
\PrintBackRefs{\CurrentBib}

\bibitem [\protect \citeauthoryear {%
van Zuilekom%
, van Maarseveen%
\BCBL {}\ \BBA {} van~der Doef%
}{%
van Zuilekom%
\ \protect \BOthers {.}}{%
{\protect \APACyear {2005}}%
}]{%
van2005decision}
\APACinsertmetastar {%
van2005decision}%
\begin{APACrefauthors}%
van Zuilekom, K.%
, van Maarseveen, M.%
\BCBL {}\ \BBA {} van~der Doef, M.%
\end{APACrefauthors}%
\unskip\
\newblock
\APACrefYearMonthDay{2005}{}{}.
\newblock
{\BBOQ}\APACrefatitle {A decision support system for preventive evacuation of
  people} {A decision support system for preventive evacuation of
  people}.{\BBCQ}
\newblock
\BIn{} \APACrefbtitle {Geo-information for disaster management}
  {Geo-information for disaster management}\ (\BPGS\ 229--253).
\newblock
\APACaddressPublisher{}{Springer}.
\PrintBackRefs{\CurrentBib}

\bibitem [\protect \citeauthoryear {%
Veliz%
, Watson%
, Weintraub%
, Wets%
\BCBL {}\ \BBA {} Woodruff%
}{%
Veliz%
\ \protect \BOthers {.}}{%
{\protect \APACyear {2015}}%
}]{%
veliz2015stochastic}
\APACinsertmetastar {%
veliz2015stochastic}%
\begin{APACrefauthors}%
Veliz, F\BPBI B.%
, Watson, J\BHBI P.%
, Weintraub, A.%
, Wets, R\BPBI J\BHBI B.%
\BCBL {}\ \BBA {} Woodruff, D\BPBI L.%
\end{APACrefauthors}%
\unskip\
\newblock
\APACrefYearMonthDay{2015}{}{}.
\newblock
{\BBOQ}\APACrefatitle {Stochastic optimization models in forest planning: a
  progressive hedging solution approach} {Stochastic optimization models in
  forest planning: a progressive hedging solution approach}.{\BBCQ}
\newblock
\APACjournalVolNumPages{Annals of Operations Research}{232}{1}{259--274}.
\PrintBackRefs{\CurrentBib}

\bibitem [\protect \citeauthoryear {%
{Virtanen}%
\ \protect \BOthers {.}}{%
{Virtanen}%
\ \protect \BOthers {.}}{%
{\protect \APACyear {2020}}%
}]{%
2020SciPy-NMeth}
\APACinsertmetastar {%
2020SciPy-NMeth}%
\begin{APACrefauthors}%
{Virtanen}, P.%
, {Gommers}, R.%
, {Oliphant}, T\BPBI E.%
, {Haberland}, M.%
, {Reddy}, T.%
, {Cournapeau}, D.%
\BDBL {}{Contributors}, S\BPBI \BPBI .%
\end{APACrefauthors}%
\unskip\
\newblock
\APACrefYearMonthDay{2020}{}{}.
\newblock
{\BBOQ}\APACrefatitle {{SciPy 1.0: Fundamental Algorithms for Scientific
  Computing in Python}} {{SciPy 1.0: Fundamental Algorithms for Scientific
  Computing in Python}}.{\BBCQ}
\newblock
\APACjournalVolNumPages{Nature Methods}{17}{}{261--272}.
\newblock
\begin{APACrefDOI} \doi{\url{https://doi.org/10.1038/s41592-019-0686-2}}
  \end{APACrefDOI}
\PrintBackRefs{\CurrentBib}

\bibitem [\protect \citeauthoryear {%
W{\"a}chter%
\ \BBA {} Biegler%
}{%
W{\"a}chter%
\ \BBA {} Biegler%
}{%
{\protect \APACyear {2006}}%
}]{%
wachter2006implementation}
\APACinsertmetastar {%
wachter2006implementation}%
\begin{APACrefauthors}%
W{\"a}chter, A.%
\BCBT {}\ \BBA {} Biegler, L\BPBI T.%
\end{APACrefauthors}%
\unskip\
\newblock
\APACrefYearMonthDay{2006}{}{}.
\newblock
{\BBOQ}\APACrefatitle {On the implementation of an interior-point filter
  line-search algorithm for large-scale nonlinear programming} {On the
  implementation of an interior-point filter line-search algorithm for
  large-scale nonlinear programming}.{\BBCQ}
\newblock
\APACjournalVolNumPages{Mathematical programming}{106}{1}{25--57}.
\PrintBackRefs{\CurrentBib}

\bibitem [\protect \citeauthoryear {%
Wallace%
\ \BBA {} De~Balogh%
}{%
Wallace%
\ \BBA {} De~Balogh%
}{%
{\protect \APACyear {1985}}%
}]{%
wallace1985decision}
\APACinsertmetastar {%
wallace1985decision}%
\begin{APACrefauthors}%
Wallace, W\BPBI A.%
\BCBT {}\ \BBA {} De~Balogh, F.%
\end{APACrefauthors}%
\unskip\
\newblock
\APACrefYearMonthDay{1985}{}{}.
\newblock
{\BBOQ}\APACrefatitle {Decision support systems for disaster management}
  {Decision support systems for disaster management}.{\BBCQ}
\newblock
\APACjournalVolNumPages{Public Administration Review}{}{}{134--146}.
\PrintBackRefs{\CurrentBib}

\bibitem [\protect \citeauthoryear {%
Walt%
, Colbert%
\BCBL {}\ \BBA {} Varoquaux%
}{%
Walt%
\ \protect \BOthers {.}}{%
{\protect \APACyear {2011}}%
}]{%
numpy_citation_2}
\APACinsertmetastar {%
numpy_citation_2}%
\begin{APACrefauthors}%
Walt, S\BPBI v\BPBI d.%
, Colbert, S\BPBI C.%
\BCBL {}\ \BBA {} Varoquaux, G.%
\end{APACrefauthors}%
\unskip\
\newblock
\APACrefYearMonthDay{2011}{}{}.
\newblock
{\BBOQ}\APACrefatitle {The NumPy array: a structure for efficient numerical
  computation} {The numpy array: a structure for efficient numerical
  computation}.{\BBCQ}
\newblock
\APACjournalVolNumPages{Computing in Science \& Engineering}{13}{2}{22--30}.
\PrintBackRefs{\CurrentBib}

\bibitem [\protect \citeauthoryear {%
Watson%
\ \BBA {} Woodruff%
}{%
Watson%
\ \BBA {} Woodruff%
}{%
{\protect \APACyear {2011}}%
}]{%
watson2011progressive}
\APACinsertmetastar {%
watson2011progressive}%
\begin{APACrefauthors}%
Watson, J\BHBI P.%
\BCBT {}\ \BBA {} Woodruff, D\BPBI L.%
\end{APACrefauthors}%
\unskip\
\newblock
\APACrefYearMonthDay{2011}{}{}.
\newblock
{\BBOQ}\APACrefatitle {Progressive hedging innovations for a class of
  stochastic mixed-integer resource allocation problems} {Progressive hedging
  innovations for a class of stochastic mixed-integer resource allocation
  problems}.{\BBCQ}
\newblock
\APACjournalVolNumPages{Computational Management Science}{8}{4}{355--370}.
\PrintBackRefs{\CurrentBib}

\bibitem [\protect \citeauthoryear {%
Watson%
, Woodruff%
\BCBL {}\ \BBA {} Hart%
}{%
Watson%
\ \protect \BOthers {.}}{%
{\protect \APACyear {2012}}%
}]{%
watson2012pysp}
\APACinsertmetastar {%
watson2012pysp}%
\begin{APACrefauthors}%
Watson, J\BHBI P.%
, Woodruff, D\BPBI L.%
\BCBL {}\ \BBA {} Hart, W\BPBI E.%
\end{APACrefauthors}%
\unskip\
\newblock
\APACrefYearMonthDay{2012}{}{}.
\newblock
{\BBOQ}\APACrefatitle {PySP: modeling and solving stochastic programs in
  Python} {Pysp: modeling and solving stochastic programs in python}.{\BBCQ}
\newblock
\APACjournalVolNumPages{Mathematical Programming Computation}{4}{2}{109--149}.
\PrintBackRefs{\CurrentBib}

\bibitem [\protect \citeauthoryear {%
Winston%
}{%
Winston%
}{%
{\protect \APACyear {1994}}%
}]{%
enme741_textbook}
\APACinsertmetastar {%
enme741_textbook}%
\begin{APACrefauthors}%
Winston, W\BPBI L.%
\end{APACrefauthors}%
\unskip\
\newblock
\APACrefYear{1994}.
\newblock
\APACrefbtitle {Operations Research: Applications and Algorithms} {Operations
  research: Applications and algorithms}.
\newblock
\APACaddressPublisher{Belmont, CA}{Wadsworth Inc.}
\PrintBackRefs{\CurrentBib}

\bibitem [\protect \citeauthoryear {%
{Wolfram|Alpha}%
}{%
{Wolfram|Alpha}%
}{%
{\protect \APACyear {2021}}%
}]{%
wolfram_alpha}
\APACinsertmetastar {%
wolfram_alpha}%
\begin{APACrefauthors}%
{Wolfram|Alpha}.%
\end{APACrefauthors}%
\unskip\
\newblock
\APACrefYearMonthDay{2021}{}{}.
\newblock
\APACrefbtitle {WolframAlpha Computational Intelligence.} {Wolframalpha
  computational intelligence.}
\newblock
\APAChowpublished {\url{https://www.wolframalpha.com/}}.
\PrintBackRefs{\CurrentBib}

\bibitem [\protect \citeauthoryear {%
Wong%
\ \BBA {} Wong%
}{%
Wong%
\ \BBA {} Wong%
}{%
{\protect \APACyear {2016}}%
}]{%
wong2016network}
\APACinsertmetastar {%
wong2016network}%
\begin{APACrefauthors}%
Wong, W.%
\BCBT {}\ \BBA {} Wong, S.%
\end{APACrefauthors}%
\unskip\
\newblock
\APACrefYearMonthDay{2016}{}{}.
\newblock
{\BBOQ}\APACrefatitle {Network topological effects on the macroscopic Bureau of
  Public Roads function} {Network topological effects on the macroscopic bureau
  of public roads function}.{\BBCQ}
\newblock
\APACjournalVolNumPages{Transportmetrica A: Transport
  Science}{12}{3}{272--296}.
\PrintBackRefs{\CurrentBib}

\bibitem [\protect \citeauthoryear {%
Yang%
, Guo%
\BCBL {}\ \BBA {} Yang%
}{%
Yang%
\ \protect \BOthers {.}}{%
{\protect \APACyear {2019}}%
}]{%
yang2019emergency}
\APACinsertmetastar {%
yang2019emergency}%
\begin{APACrefauthors}%
Yang, Z.%
, Guo, L.%
\BCBL {}\ \BBA {} Yang, Z.%
\end{APACrefauthors}%
\unskip\
\newblock
\APACrefYearMonthDay{2019}{}{}.
\newblock
{\BBOQ}\APACrefatitle {Emergency logistics for wildfire suppression based on
  forecasted disaster evolution} {Emergency logistics for wildfire suppression
  based on forecasted disaster evolution}.{\BBCQ}
\newblock
\APACjournalVolNumPages{Annals of Operations Research}{283}{1}{917--937}.
\PrintBackRefs{\CurrentBib}

\bibitem [\protect \citeauthoryear {%
Yilmaz%
, Aydemir-Karadag%
\BCBL {}\ \BBA {} Erol%
}{%
Yilmaz%
\ \protect \BOthers {.}}{%
{\protect \APACyear {2019}}%
}]{%
yilmaz2019finding}
\APACinsertmetastar {%
yilmaz2019finding}%
\begin{APACrefauthors}%
Yilmaz, Z.%
, Aydemir-Karadag, A.%
\BCBL {}\ \BBA {} Erol, S.%
\end{APACrefauthors}%
\unskip\
\newblock
\APACrefYearMonthDay{2019}{}{}.
\newblock
{\BBOQ}\APACrefatitle {Finding optimal depots and routes in sudden-onset
  disasters: An earthquake case for Erzincan} {Finding optimal depots and
  routes in sudden-onset disasters: An earthquake case for erzincan}.{\BBCQ}
\newblock
\APACjournalVolNumPages{Transportation journal}{58}{3}{168--196}.
\PrintBackRefs{\CurrentBib}

\bibitem [\protect \citeauthoryear {%
H.~Zhang%
\ \BBA {} Ritchie%
}{%
H.~Zhang%
\ \BBA {} Ritchie%
}{%
{\protect \APACyear {1994}}%
}]{%
zhang1994real}
\APACinsertmetastar {%
zhang1994real}%
\begin{APACrefauthors}%
Zhang, H.%
\BCBT {}\ \BBA {} Ritchie, S\BPBI G.%
\end{APACrefauthors}%
\unskip\
\newblock
\APACrefYearMonthDay{1994}{}{}.
\newblock
{\BBOQ}\APACrefatitle {Real-time decision-support system for freeway management
  and control} {Real-time decision-support system for freeway management and
  control}.{\BBCQ}
\newblock
\APACjournalVolNumPages{Journal of Computing in Civil
  Engineering}{8}{1}{35--51}.
\PrintBackRefs{\CurrentBib}

\bibitem [\protect \citeauthoryear {%
J.~Zhang%
\ \BBA {} Paschalidis%
}{%
J.~Zhang%
\ \BBA {} Paschalidis%
}{%
{\protect \APACyear {2017}}%
}]{%
zhang2017data}
\APACinsertmetastar {%
zhang2017data}%
\begin{APACrefauthors}%
Zhang, J.%
\BCBT {}\ \BBA {} Paschalidis, I\BPBI C.%
\end{APACrefauthors}%
\unskip\
\newblock
\APACrefYearMonthDay{2017}{}{}.
\newblock
{\BBOQ}\APACrefatitle {Data-driven estimation of travel latency cost functions
  via inverse optimization in multi-class transportation networks} {Data-driven
  estimation of travel latency cost functions via inverse optimization in
  multi-class transportation networks}.{\BBCQ}
\newblock
\BIn{} \APACrefbtitle {2017 IEEE 56th Annual Conference on Decision and Control
  (CDC)} {2017 ieee 56th annual conference on decision and control (cdc)}\
  (\BPGS\ 6295--6300).
\PrintBackRefs{\CurrentBib}

\bibitem [\protect \citeauthoryear {%
J.~Zhang%
, Pourazarm%
, Cassandras%
\BCBL {}\ \BBA {} Paschalidis%
}{%
J.~Zhang%
\ \protect \BOthers {.}}{%
{\protect \APACyear {2018}}%
}]{%
zhang2018price}
\APACinsertmetastar {%
zhang2018price}%
\begin{APACrefauthors}%
Zhang, J.%
, Pourazarm, S.%
, Cassandras, C\BPBI G.%
\BCBL {}\ \BBA {} Paschalidis, I\BPBI C.%
\end{APACrefauthors}%
\unskip\
\newblock
\APACrefYearMonthDay{2018}{}{}.
\newblock
{\BBOQ}\APACrefatitle {The price of anarchy in transportation networks:
  Data-driven evaluation and reduction strategies} {The price of anarchy in
  transportation networks: Data-driven evaluation and reduction
  strategies}.{\BBCQ}
\newblock
\APACjournalVolNumPages{Proceedings of the IEEE}{106}{4}{538--553}.
\PrintBackRefs{\CurrentBib}

\bibitem [\protect \citeauthoryear {%
J\BHBI z.~Zhang%
, Jian%
\BCBL {}\ \BBA {} Tang%
}{%
J\BHBI z.~Zhang%
\ \protect \BOthers {.}}{%
{\protect \APACyear {2011}}%
}]{%
zhang2011inverse}
\APACinsertmetastar {%
zhang2011inverse}%
\begin{APACrefauthors}%
Zhang, J\BHBI z.%
, Jian, J\BHBI b.%
\BCBL {}\ \BBA {} Tang, C\BHBI m.%
\end{APACrefauthors}%
\unskip\
\newblock
\APACrefYearMonthDay{2011}{}{}.
\newblock
{\BBOQ}\APACrefatitle {Inverse problems and solution methods for a class of
  nonlinear complementarity problems} {Inverse problems and solution methods
  for a class of nonlinear complementarity problems}.{\BBCQ}
\newblock
\APACjournalVolNumPages{Computational Optimization and
  Applications}{49}{2}{271--297}.
\PrintBackRefs{\CurrentBib}

\bibitem [\protect \citeauthoryear {%
Zheng%
\ \BBA {} Ling%
}{%
Zheng%
\ \BBA {} Ling%
}{%
{\protect \APACyear {2013}}%
}]{%
zheng2013emergency}
\APACinsertmetastar {%
zheng2013emergency}%
\begin{APACrefauthors}%
Zheng, Y\BHBI J.%
\BCBT {}\ \BBA {} Ling, H\BHBI F.%
\end{APACrefauthors}%
\unskip\
\newblock
\APACrefYearMonthDay{2013}{}{}.
\newblock
{\BBOQ}\APACrefatitle {Emergency transportation planning in disaster relief
  supply chain management: a cooperative fuzzy optimization approach}
  {Emergency transportation planning in disaster relief supply chain
  management: a cooperative fuzzy optimization approach}.{\BBCQ}
\newblock
\APACjournalVolNumPages{Soft Computing}{17}{7}{1301--1314}.
\PrintBackRefs{\CurrentBib}

\end{thebibliography}
\nocite{}

\appendix

\numberwithin{equation}{section}
\numberwithin{table}{section}
\numberwithin{figure}{section}

\section{Data Analysis Component: Inverse Optimization}

\subsection{Proof of $\mathbf{d}^w - N \mathbf{x}^w = 0$}

\begin{lemma} 
For the following complementarity problem:
\begin{subequations}\label{traffic_equilibrium_problem_appendix}
\begin{equation}\label{TEP:cost_appendix}
    0 \leq \mathbf{c}(\mathbf{x}) + N^T \mathbf{y}^w \ \bot \ \mathbf{x}^w \geq 0,\  \forall w \in \mathcal{W}
\end{equation}
\begin{equation}\label{TEP:demand_appendix}
     0 \leq \mathbf{d}^w - N \mathbf{x}^w\  \bot\  \mathbf{y}^w \geq 0,\ \forall w \in \mathcal{W}
\end{equation}
\end{subequations}

\noindent $\mathbf{d}^w - N \mathbf{x}^w = 0$ when there is a solution for (\ref{traffic_equilibrium_problem_appendix}) and when we assume that the $\mathbf{c}(\mathbf{x})$ function is greater than 0 for all $\mathbf{x}\geq 0$.
\end{lemma}

\begin{proof}
This proof is adapted from a proof seen in \cite{ban2005quasi}.  Assume that the $\mathbf{c}(\mathbf{x})$ function is greater than 0 for all $\mathbf{x} \geq 0$.  Also assume for the sake of contradiction that 
\begin{equation}\label{proof-demand:equation_1}
0 < d_i^w - \left(\sum\limits_{l:(l,i)}x_{(l,i)}^w - \sum\limits_{j:(i,j)} x_{(i,j)}^w \right) 
\end{equation} 

\noindent for some destination $w \in \mathcal{W}$ and for some $i \in \mathcal{N}$.  $\sum\limits_{l:(l,i)}x_{(l,i)}^w$ represents the inflow at node $i$, and $\sum\limits_{j:(i,j)} x_{(i,j)}^w$ represents the outflow at node $i$.  We know $y_i^w = 0$ by complementarity in (\ref{TEP:demand_appendix}).  We can rearrange the inequality in (\ref{proof-demand:equation_1}) to say:

\begin{equation}\label{proof-demand:equation_2}
0 \leq \sum\limits_{l:(l,i)}x_{(l,i)}^w < d_i^w +  \sum\limits_{j:(i,j)} x_{(i,j)}^w  
\end{equation} 

\noindent for some $w \in \mathcal{W}$ (representing the final destination) and for some $i \in \mathcal{N}$. We have that the $\sum\limits_{l:(l,i)}x_{(l,i)}^w$ term is greater than or equal to 0 because we know all $\mathbf{x}^w \geq 0$.  The inequality in (\ref{proof-demand:equation_2}) produces three difference cases:
\begin{itemize}
    \item Case 1: $d_i^w = 0$.  This means at least one link in the $\sum\limits_{j:(i,j)} x_{(i,j)}^w$ sum must be positive because $0 < d_i^w +  \sum\limits_{j:(i,j)} x_{(i,j)}^w$.  Therefore, for such a link $(i,j)$, $x_{(i,j)}^w > 0$ forces the following equality:
    \begin{equation}
    \mathbf{c}_{(i,j)}(\mathbf{x}) + y_j^w - y_i^w = 0
    \end{equation}
    
    \noindent We know $y_i^w = 0$, so we have $\mathbf{c}_{(i,j)}(\mathbf{x}) + y_j^w = 0$.  Since $y_j^w \geq 0$, both components must be zero but that contradicts the assumption that $\mathbf{c}(\mathbf{x})$ is greater than 0 for all $\mathbf{x} \geq 0$, thereby contradicting (\ref{proof-demand:equation_1}).
    
    \item Case 2: $d_i^w$ is negative.  This means at least one $x^w_{(i,j)} > 0$ and contradiction follows as in Case 1.
    
    \item Case 3: $d_i^w$ is greater than 0.  This implies that $i=w$ because only the final destination has a positive value.  There are some sub-cases to this case, but we first note that we know for a general node $k \neq i$:
    \begin{equation}\label{knowledge_gained}
        \sum\limits_{l:(l,k)}x_{(l,k)}^w - \sum\limits_{j:(k,j)} x_{(k,j)}^w = d_k^w 
    \end{equation}
    
    \noindent for the cases of $d_k^w = 0$ or when $d_k^w$ is negative, based on Cases 1 and 2.
    
    \begin{itemize}
        \item Sub-Case A: $\sum\limits_{j:(i,j)} x_{(i,j)}^w > 0$.  In this case, we arrive at the same contradictions we arrived at for the previous two cases.
        
        \item Sub-Case B: $\sum\limits_{j:(i,j)} x_{(i,j)}^w = 0$ and $\sum\limits_{l:(l,i)}x_{(l,i)}^w = 0$. In this case, we know there is some $d_k^w$ that is negative, which by (\ref{knowledge_gained}) means $\sum\limits_{j:(k,j)} x_{(k,j)}^w > 0$.  For any connecting nodes $q$ between node $k$ and node $i$, $\sum\limits_{l:(l,q)}x_{(l,q)}^w = \sum\limits_{j:(q,j)} x_{(q,j)}^w > 0$.  Therefore, for node $i$, the inflow sum $\sum\limits_{l:(l,i)}x_{(l,i)}^w$ must be greater than 0, contradicting our assumption. Overall, we arrive at the contradiction because flow begins at a node, which produces outflow to neighboring nodes, and this in turn produces inflow at node $i$.
        
        \item Sub-Case C: $\sum\limits_{j:(i,j)} x_{(i,j)}^w = 0$ and $\sum\limits_{l:(l,i)}x_{(l,i)}^w > 0$ but less than $d_i^w$.  For any $k$ nodes in which $d_k^w$ is negative, we know the absolute sum over these $k$ nodes is equal to $d_i^w$ when $i=w$.  Therefore, as in Sub-Case B, there are outflows at these $k$ nodes due to the relationship (\ref{knowledge_gained}).  There is also conservation of flow at the $q$ nodes in which $d_q^w = 0$.  Therefore, for the $q$ nodes connected to node $i$, the outflow from those nodes must match the inflow from previous nodes, which if taken back to the $k$ nodes, would equal the total $d_i^w$ sum.  Therefore, for node $i$, the inflow sum $\sum\limits_{l:(l,i)}x_{(l,i)}^w$ must be equal to $d_i^w$.  This contradicts our assumption that the inflow would be less than $d_i^w$.  Overall, we arrive at the contradiction because flow would be pushed toward the $i$ destination in order to satisfy the relationships established by (\ref{knowledge_gained}).
    \end{itemize}
    
\end{itemize}

\noindent Consequently, we have shown that a solution to the traffic equilibrium problem will result in the $\mathbf{d}^w - N \mathbf{x}^w = 0$ if we assume the $\mathbf{c}(\mathbf{x})$ function is greater than 0 for all $\mathbf{x} \geq 0$.
\end{proof}

\subsection{Explanation of Forming the Inverse Model from \cite{bertsimas2015data}}

To form the inverse optimization mathematical program from \cite{bertsimas2015data} for our traffic equilibrium problem, we return to the VI formulation of the problem and notice the structure (\ref{eqn:begin_min_problem_appendix})
\begin{subequations}
\begin{equation}\label{eqn:begin_min_problem_appendix}
    \mathbf{c}(\mathbf{x}^*)^T \mathbf{x} \geq \mathbf{c}(\mathbf{x}^*)^T \mathbf{x}^* - \epsilon,\ \forall \mathbf{x} \in \mathcal{F}
\end{equation}
\begin{equation}
    \mathcal{F} = \left\{ \mathbf{x} : \mathbf{x} \in \mathbb{R}^{|\mathcal{A}|}_+\  s.t.\ N\mathbf{x} \leq  \mathbf{d} \right\}
\end{equation}
\end{subequations}

\noindent Note, the $\mathcal{F}$ set is written slightly differently here than how it was initially introduced in equation (2) in Section 3.1 in order to mirror the complementarity problem (\ref{traffic_equilibrium_problem_appendix}) and in order to represent the fact that we only are working with one destination at a time (hence we do not need the $w$ index).  We notice that we can turn the left hand side of the (\ref{eqn:begin_min_problem_appendix}) inequality into a minimization problem
\begin{equation}\label{LP_problem_appendix}
    \min\limits_{\mathbf{x}\in \mathcal{F}} \mathbf{c}(\mathbf{x}^*)^T \mathbf{x}
\end{equation}
in which $\mathbf{x}^*$ is fixed, and this minimization problem forms the tightest upper bound on the right hand side of (\ref{eqn:begin_min_problem_appendix}) since we are choosing the $\mathbf{x}$ to minimize the left hand side of (\ref{eqn:begin_min_problem_appendix}) [\cite{facchinei2007finite,bertsimas2015data}].  Because (\ref{LP_problem_appendix}) is a linear program, we know strong duality holds, which means   
we can find the dual of this problem and know that there is no duality gap between the primal and the dual [\cite{murty1983linear,enme741_textbook}]. The dual of this problem is:
\begin{subequations}
\begin{equation}
    \max\limits_{\mathbf{y}^w} (-\mathbf{d}^w)^T \mathbf{y}^w
\end{equation}
\begin{equation}
    -N^T \mathbf{y}^w \leq \mathbf{c}(\mathbf{x}^*)
\end{equation}
\begin{equation}
    \mathbf{y}^w \geq 0
\end{equation}
\end{subequations}

\noindent As in \cite{bertsimas2015data}, we then equate the dual objective and the primal objective, and our final set of constraints representing the satisfaction of the variational inequality in (\ref{eqn:begin_min_problem_appendix}) are:
\begin{subequations}
\begin{equation}\label{dual_primal_appendix}
    \mathbf{c}(\mathbf{x}^*)^T \mathbf{x}^* + (\mathbf{d}^w)^T \mathbf{y} \leq \epsilon
\end{equation}
\begin{equation}\label{dual_feas_1_appendix}
    -N^T \mathbf{y}^w \leq \mathbf{c}(\mathbf{x}^*)
\end{equation}
\begin{equation}\label{dual_feas_2_appendix}
    \mathbf{y}^w \geq 0
\end{equation}
\end{subequations}

\noindent which include the equating of the dual and primal objectives (\ref{dual_primal_appendix}) as well as the dual feasibility constraints (\ref{dual_feas_1_appendix}-\ref{dual_feas_2_appendix}). Using these conditions, we can then form an optimization problem including each data point $\hat{\mathbf{x}}^j$ representing the flow on the network such that: 

\begin{itemize}
\item There is one OD pair for each instance $\hat{\mathbf{x}}^j$. 

\item There is the same node-arc incidence matrix $N$ for each $\hat{\mathbf{x}}^j$.
\end{itemize}

\noindent The optimization problem becomes for $J$ data points, parameters $\theta \in \Theta$ with $\Theta$ as a convex subset of $\mathbb{R}^{Z}$ ($Z$ representing a number of parameters), $\mathbf{y}^j \in \mathbb{R}^{|\mathcal{N}|}$, and $\epsilon \in \mathbb{R}^{J}$:  

\begin{subequations}\label{basic_bertsimas_formulation_appendix}
\begin{equation}
    \min\limits_{\theta \in \Theta,\mathbf{y},\epsilon} ||\epsilon||^2
\end{equation}
\begin{equation}
    -(N)^T \mathbf{y}^j \leq \mathbf{c}(\hat{\mathbf{x}}^j;\theta),\ j=1,...,J,
\end{equation}
\begin{equation}
    \mathbf{y}^j \geq 0,\ j = 1,...,J,
\end{equation}
\begin{equation}
    \mathbf{c}(\hat{\mathbf{x}}^j;\theta)^T \hat{\mathbf{x}}^j + (\mathbf{d}^j)^T \mathbf{y}^j \leq \epsilon^j,\ j=1,...,J,
\end{equation}
\end{subequations}

\noindent The set $\Theta$ is determined by the lower and upper bounds on the parameter values found in Table 1.

\section{Normative Model: Two-Stage Stochastic Model}

\subsection{Two-Stage Stochastic Model: Parameter Details}

We adopt the notation from \cite{fan2010solving}, with the exception of the parameters in the $h_{a}^s(u_a)$ function which, although inspired by \cite{fan2010solving}, is of a different form:

\begin{itemize}
    \item $\mathcal{A}$: the set of network arcs, and $m$ as the number of arcs.
    
    \item $\mathcal{N}$: the set of network nodes, and $n$ as the number of nodes.
    
    \item $K$: the number of destinations of flow in the network. 
    
    \item $\mathcal{S}$: the scenario set
    
    \item $x_a^{k,s}$: the flow on arc $a$ that is destined for the $k$th destination in scenario $s$.  The vector $\mathbf{x}^{k,s} \in \mathbb{R}^m$ denotes the flow on all arcs that is headed for the $k$th destination in scenario $s$. (units=thousands of vehicles)
    
    \item $f_{a}^s$: the total flow on arc $a$ in scenario $s$, and $\mathbf{f}^s$ as the vector containing all of the $f_a^s$ decision variables for scenario $s$. (units=thousands of vehicles)
    
    \item $u_a$: the decision variable controlling resources used to protect an arc $a$ against a crisis.  
    Some examples of potential protection decisions include protective measures against landslides and flash floods as in the case of Nepal [\cite{WFP_link}]. (units=proportion of necessary resources needed to fully insure the arc) 
    
    \item $W$: the node-link adjacency matrix. We use the definition from \cite{marcotte2007traffic}'s work of this matrix which states $W \in \{-1,0,1 \}^{|N|\times|A|}$ such that, for a given column (representing an arc), there is a -1 at the node in which the arc begins and a 1 at the node in which the arc ends. 
    
    
    \item $\mathbf{q}^k \in \mathbb{R}^n$: designates the amount of flow originating at each node that is headed to destination $k$.  We based our construction of the $\mathbf{q}^k$ vectors upon the set-up from \cite{marcotte2007traffic}'s such that negative entries within the vector indicate the presence of and amount of demand at those nodes and such that a single positive entry denotes location of the demand (and is the absolute sum of the negative entries).  
    (units = thousands of vehicles)
    
    \item $h_{a}^s(u_a)$: the capacity of an arc $a$ given first stage decision $u_a$ under scenario $s$:
    \begin{equation}
        h_a^s(u_a) = \begin{cases} \text{cap}_a\ \text{ if } a \notin \bar{\mathcal{A}} \\ \text{cap}_a - m_a^s (1-u_a)\ \text{ if } a \in \bar{\mathcal{A}} \end{cases}
    \end{equation}
    
    \noindent with $\text{cap}_a$ representing capacity of the arc without it being affected by a disaster, $m_a^s$ representing the amount of damage done to arc $a$ in scenario $s$ if not protected, and $\bar{\mathcal{A}}$ represents the set of arc vulnerable to the disaster.  Note that $m_a^s$ could be 0 in certain scenarios.  (units = thousands of vehicles)
    

    \item $t_a(\mathbf{f}^s)$ represents the time per vehicle along arc $a$ \cite{lu2018mean,lu2016bi} as a function of the flows $\mathbf{f}^s$ in scenario $s$.  We explore multiple different forms for $t_a$.   
    
    \item $\lambda_i^{k,s}$ as ``the minimum time from node $i$ to destination $k$'' in scenario $s$ according to \cite{fan2010solving}. (units=travel time)
    
    \item $\mathbf{d}^{k,s}$ as the vector of extra variables that acts as a buffer for any flow that cannot be properly apportioned. (units=thousands of vehicles).
    
    \item $p^s$ as the probability of each scenario $s$
    
    
\end{itemize}

\subsection{Calculating the $M_{i,j}^{k,s}$ Values}

The $M_{i,j}^{k,s}$ values are the numbers utilized in the disjunctive constraints in Section 3.2.1. To calculate the $M_{i,j}^{k,s}$ values, we use the following reasoning.  First, we know that the maximum flow on a given arc is 8.  We also know from Table 1 that the maximum value of $\phi_a$ and $\beta_a$ is 10.  Therefore, we input $x_a=8$ into $\phi_a x_a + \beta_a$, obtain 90, and then multiply by the number of arcs to obtain an upper bound, which can be increased if desired.  We decide to increase it by multiplying by 2.  The resulting value represents an upper bound on the maximum travel time between an origin and destination point in the networks under the linear cost function.  It also works for the BPR cost function because if we take the maximum value of that function for a given arc, we would get 12, which is significantly below 90.  The final value of $M_{i,j}^{k,s}$ is $90(m)(2)$, with $m$ as the number of arcs.

\section{Results}

\subsection{Flow Error under IO $\alpha$}

\cite{allen2021using} define a flow error metric to evaluate whether or not an IO parameterization is valid for their application.  Since the $\alpha$ values (for the BPR functions) imputed through IO are different from the original $\alpha$ values, we evaluate the flow error for each network used.  This flow error metric is the Frobenius norm between the flow values across all the arcs for all of the OD pairs in a given trial.  The flow error metrics for the two networks can be seen in Table \ref{flow_error} below:

\begin{table}[H]
    \centering
    \begin{tabular}{c|c}
        \textbf{4x4 Grid (Experiment II)} & \textbf{Nguyen \& Dupuis Network (Experiment IV)} \\ \hline
        0.0002 & 1.57e-05\\
        0.0001 & 2.93e-05 \\
        0.0002 & 7.11e-05 \\
        0.0002 & 3.08e-05\\
        5.73e-05 & 8.62e-05\\
        0.0007 & 0.0004\\
        0.0001 & 4.02e-05\\
        6.39e-05 & 4.43e-05 \\
        0.0001 & 2.26e-05 \\
        7.82e-05 & 6.06e-05 \\
    \end{tabular}
    \caption{Flow Errors for BPR Functions on the Two Networks}
    \label{flow_error}
\end{table}

\noindent From the small magnitude of these values, we see that the $\alpha$ recovered by the IO model produce flow values that are very close to the flow values produced by the original $\alpha$ values.

\subsection{Median/Min-Max Tables and Nguyen \& Dupuis Boxplots}

When examining the medians as a percentage of the budget for Experiments I-III in Tables \ref{table:medians_ranges_epsilon_0_01} and \ref{table:medians_ranges_epsilon_0_001}, the O-IO metric medians are quire small compared to the U-IO \& U-O metric medians, thus again supporting the claim that IO can be used to recover the original cost protection decisions and that the protection decisions made under IO and original costs are different from the protection decisions made under uniform cost.  

Looking at Figure \ref{fig:experiment_3_metrics}, the metric data are not overlapping which supports the idea that the protection decisions under IO or original costs differ when compared to the protection decisions under uniform or baseline cost parameters.  In Figure \ref{fig:experiment_4_metrics}, we see that the decisions under uniform cost do not differ from the IO imputed cost decisions in Experiment IV as much as in other experiments.  However, this could be a result of the small interval in which $\alpha$ was allowed to vary.  In future work, it would be interesting to experiment with wider intervals to further understand this behavior.  At the same time, these results do not take away from our conclusion that IO is able to impute costs that lead to protection decisions similar to those of the original cost. 

\begin{table}[H]
\centering
\begin{footnotesize}
\begin{tabular}{c|cc|cc|cc|cc}
     & \multicolumn{2}{c}{\textbf{Experiment I}} & \multicolumn{2}{c}{\textbf{Experiment II}} & \multicolumn{2}{c}{\textbf{Experiment III}} & \multicolumn{2}{c}{\textbf{Experiment IV}} \\ \cline{2-9}
     & Med & (Min, Max) & Med & (Min, Max) & Med & (Min, Max) & Med & (Min, Max) \\ \hline
     \textbf{O-IO}& \makecell{0.0004\\0.01\%} &(0.0, 0.0765)& \makecell{0.0006\\0.01\%} &(0.0, 0.0076)& \makecell{0.0145\\0.24\%} &(0.0, 0.0799)& \makecell{0.0014\\0.02\%} &(0.0, 0.0181)\\ \hline \textbf{U-IO}& \makecell{0.347\\5.78\%} &(0.2213, 0.4441)& \makecell{0.0446\\0.74\%} &(0.0108, 0.1641)& \makecell{0.2877\\4.79\%} &(0.1629, 0.3472)& \makecell{0.0112\\0.19\%} &(0.0, 0.1819)\\ \hline \textbf{U-O}& \makecell{0.347\\5.78\%} &(0.2213, 0.4438)& \makecell{0.0447\\0.74\%} &(0.0083, 0.164)& \makecell{0.2891\\4.82\%} &(0.1502, 0.3472)& \makecell{0.004\\0.07\%} &(0.0, 0.1819)
\end{tabular}
\end{footnotesize}
\caption{Medians, Medians as Percentage of $I=6$ Budget, and Ranges for Experiments, $\epsilon=0.01$} \label{table:medians_ranges_epsilon_0_01}
\end{table}

\begin{table}[H]
\centering
\begin{footnotesize}
\begin{tabular}{c|cc|cc|cc|cc}
     & \multicolumn{2}{c}{\textbf{Experiment I}} & \multicolumn{2}{c}{\textbf{Experiment II}} & \multicolumn{2}{c}{\textbf{Experiment III}} & \multicolumn{2}{c}{\textbf{Experiment IV}} \\ \cline{2-9}
     & Med & (Min, Max) & Med & (Min, Max) & Med & (Min, Max) & Med & (Min, Max) \\ \hline
     \textbf{O-IO}& \makecell{0.0005\\0.01\%} &(0.0, 0.0765)& \makecell{0.0006\\0.01\%} &(0.0, 0.0041)& \makecell{0.0009\\0.02\%} &(0.0, 0.1323)& \makecell{0.001\\0.02\%} &(0.0, 0.0941)\\ \hline \textbf{U-IO}& \makecell{0.3147\\5.24\%} &(0.2189, 0.4747)& \makecell{0.0472\\0.79\%} &(0.0092, 0.2037)& \makecell{0.2781\\4.64\%} &(0.1629, 0.3578)& \makecell{0.0095\\0.16\%} &(0.0, 0.1209)\\ \hline \textbf{U-O}& \makecell{0.3163\\5.27\%} &(0.2189, 0.4744)& \makecell{0.0472\\0.79\%} &(0.007, 0.2032)& \makecell{0.2782\\4.64\%} &(0.1502, 0.3578)& \makecell{0.0075\\0.13\%} &(0.0, 0.1251) \\
\end{tabular}
\end{footnotesize}
\caption{Medians, Medians as Percentage of $I=6$ Budget, and Ranges for Experiments, $\epsilon=0.001$} \label{table:medians_ranges_epsilon_0_001}
\end{table}

\begin{figure}[H]
\begin{subfigure}{0.5\textwidth}
\centering
\includegraphics[height=0.2\textheight,keepaspectratio]{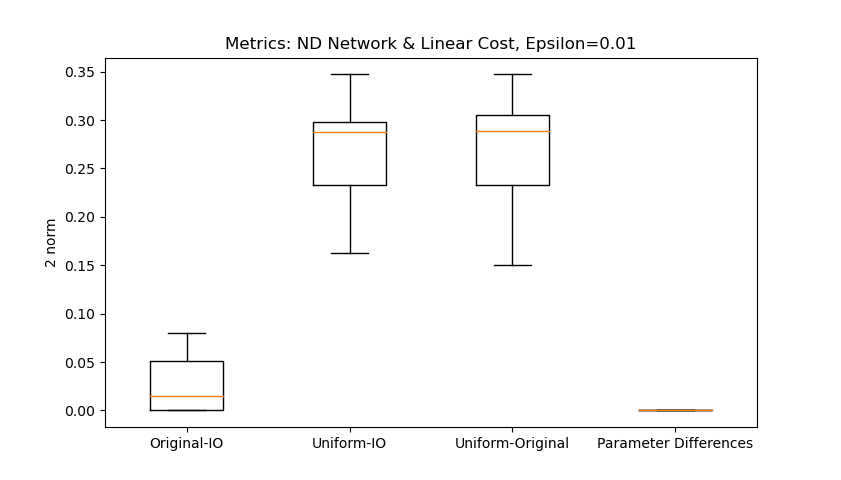}
\caption{Experiment III Results for $\epsilon=0.01$}\label{fig:experiment_3_0_01}
\end{subfigure}
\begin{subfigure}{0.5\textwidth}
\centering
\includegraphics[height=0.2\textheight,keepaspectratio]{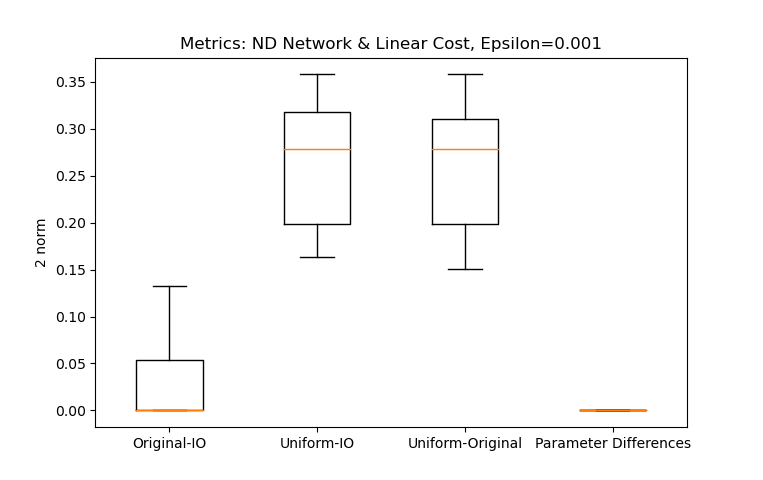}
\caption{Experiment III Results for $\epsilon=0.001$}\label{fig:experiment_3_0_001}
\end{subfigure}
\caption{Experiment III Results: Nguyen \& Dupuis Network with Linear Cost.  The parameter differences refer to the $\phi$ differences.}
\label{fig:experiment_3_metrics}
\end{figure}

\begin{figure}[H]
\begin{subfigure}{0.5\textwidth}
\centering
\includegraphics[height=0.2\textheight,keepaspectratio]{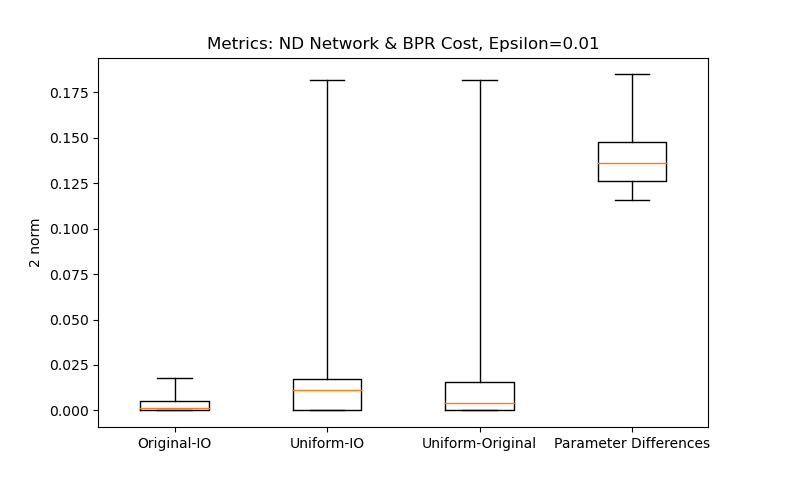}
\caption{Experiment IV Results for $\epsilon=0.01$}\label{fig:experiment_4_0_01}
\end{subfigure}
\begin{subfigure}{0.5\textwidth}
\centering
\includegraphics[height=0.2\textheight,keepaspectratio]{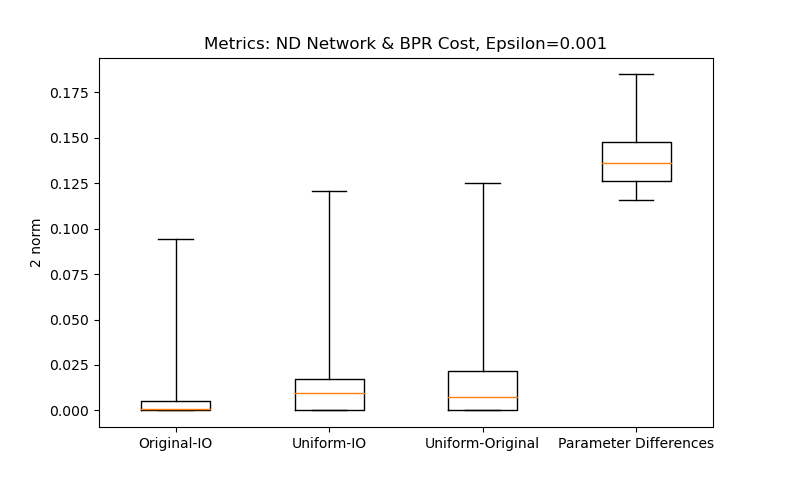}
\caption{Experiment IV Results for $\epsilon=0.001$}\label{fig:experiment_4_0_001}
\end{subfigure}
\caption{Experiment IV Results: Nguyen \& Dupuis Network with BPR.  The parameter differences refer to the $\alpha$ differences.}
\label{fig:experiment_4_metrics}
\end{figure}

\subsection{Run Time Results}

The following box-plots illustrate the run time data for Experiments I-IV and for both values of $\epsilon$, which is the value of the $g^{(k)}$ error metric in which the iterations could stop (or if 300 iterations occurred).  As a reminder, all of the experiments were run on an 8 core machine.  

For $\epsilon=0.01$, most of the trials of the experiments were below 200 minutes and, for $\epsilon=0.001$, most of the trials of the experiments were below 500 minutes.  It can be seen that there were some outliers for Experiment II, which was likely due to the additional variables needed to estimate the BPR function and the larger graph size.  

\begin{figure}[H]
\begin{subfigure}{0.5\textwidth}
\centering
\includegraphics[height=0.2\textheight,keepaspectratio]{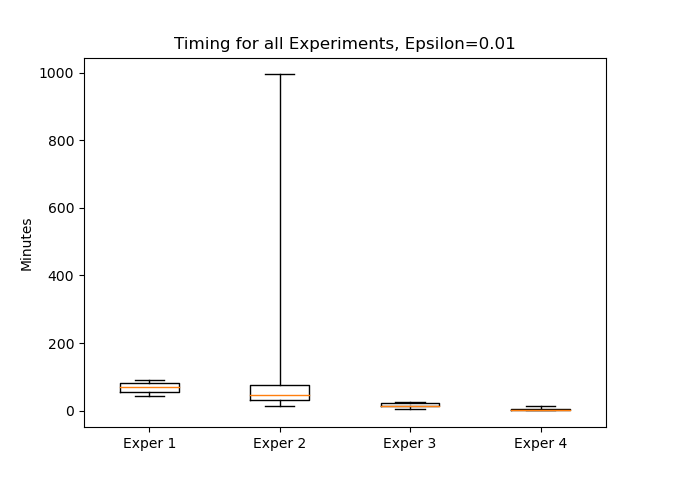}
\caption{Experiment Timing for $\epsilon=0.01$}\label{fig:timing_0_01}
\end{subfigure}
\begin{subfigure}{0.5\textwidth}
\centering
\includegraphics[height=0.2\textheight,keepaspectratio]{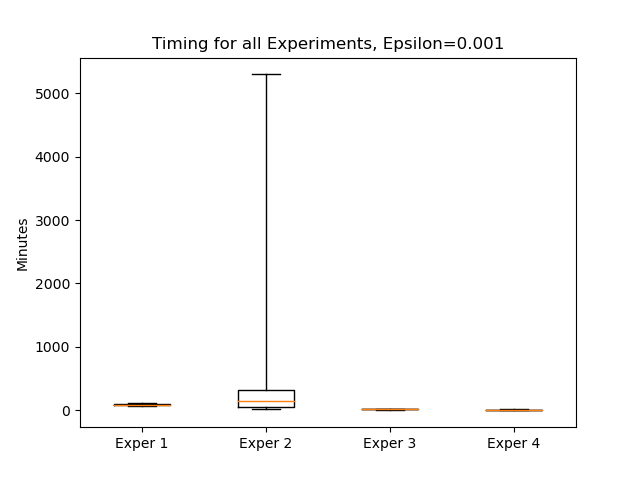}
\caption{Experiment Timing for $\epsilon=0.001$}\label{fig:timing_0_001}
\end{subfigure}
\caption{Experiment Timing Results (Minutes)}
\label{fig:timing_experiments}
\end{figure}

\section{Code Attribution}

Below are the various code resources, packages, etc. that we utilized over the course of the project:

\begin{itemize}
    \item Python (Version 3.8.5) Package
    \begin{itemize}
        \item \url{pyomo} 5.7.1 \cite{hart2011pyomo,hart2017mathematical}
        \item \url{pysp} 5.7.1 \cite{watson2012pysp}
        \item \url{networkx} 2.5 \cite{schult2008exploring}
        \item \url{pandas} 1.1.3  \cite{mckinney2010data}
        \item \url{numpy} 1.19.2 \cite{numpy_citation,numpy_citation_2,2020SciPy-NMeth}
        \item \url{scipy} 1.5.2 \cite{2020SciPy-NMeth}
        \item \url{matplotlib} 3.3.2 \cite{Hunter:2007}
    \end{itemize}
    \item Solvers
    \begin{itemize}
        \item \url{gurobi} Version 9.1.1 \cite{gurobi_citation}
        \item \url{ipopt} \cite{wachter2006implementation}
    \end{itemize}
    \item MATLAB 9.8.0.1417392 (R2020a) Update 4 \cite{MATLAB:2020a} 
    \item GAMS \cite{GAMS_software_34_1} with PATH solver \cite{dirkse1995path,ferris2000complementarity} (PATH website \cite{path_website})
    \item Wolfram Alpha \cite{wolfram_alpha}
    \item Data:
    \begin{itemize}
        \item Nguyen \& Dupuis Network \cite{nguyen1984efficient} 
    \end{itemize}
    \item Important Sites with Example Code for \url{pysp} Implementation in Scripts:
    \begin{itemize}
        \item \url{https://projects.coin-or.org/Pyomo/browser/pyomo/trunk/examples/pysp/farmer/concrete/ReferenceModel.py?rev=9358}
        
        \item \url{https://github.com/Pyomo/pysp/blob/master/examples/farmer/concreteNetX/ReferenceModel.py}
        
        \item \url{https://pyomo.readthedocs.io/en/stable/advanced_topics/pysp_rapper/demorapper.html#ph}
    \end{itemize}
    
\end{itemize}


\end{document}